\documentclass[fleqn,11pt]{article}
\usepackage[margin=2.5cm,bottom=2.65cm,footskip=1.25cm]{geometry}

\title{%
	Homometry and direct-sum decompositions of lattice-convex sets\footnote{The final publication is available at Springer via \url{http://dx.doi.org/10.1007/s00454-016-9786-2}.
		}
}
\author{%
	Gennadiy Averkov\footnote{Faculty of Mathematics, Otto-von-Guericke-Universit\"at Magdeburg, Universit\"atsplatz 2, 39106 Magdeburg, Germany, e-mail: averkov@math.uni-magdeburg.de} \, and \, 
	Barbara Langfeld\footnote{Mathematisches Seminar, Christian-Albrechts-Universit\"at zu Kiel, 24098 Kiel, Germany, e-mail: langfeld@math.uni-kiel.de}%
}

\usepackage{amsfonts,amssymb,amsthm,amsmath,graphicx,enumerate,srcltx,framed}
\usepackage{booktabs}

\usepackage[textfont=it,labelfont=bf]{caption}
\usepackage[capbesideposition=inside,facing=yes,capbesidesep=qquad]{floatrow}

\DeclareFixedFont{\ttb}{T1}{txtt}{bx}{n}{9} 
\DeclareFixedFont{\ttm}{T1}{txtt}{m}{n}{9}  
\usepackage{color}
\definecolor{deepblue}{rgb}{0,0,0.5}
\definecolor{deepred}{rgb}{0.6,0,0}
\definecolor{deepgreen}{rgb}{0,0.5,0}
\usepackage{listings}
\usepackage{tikz}\usetikzlibrary{calc,through,patterns}

\usepackage[latin1]{inputenc}

\usepackage{hyperref}%
\hypersetup{%
  pdfauthor={Gennadiy Averkov and Barbara Langfeld},%
  pdftitle={},%
  pdfsubject={},%
  linkcolor=blue, 
  urlcolor=blue, 
  citecolor=blue, 
  breaklinks=true, 
  colorlinks=true, 
  plainpages=false,%
  hypertexnames=false,%
  bookmarksopen=false,%
  bookmarksnumbered=false%
}%

\DeclareSymbolFont{rsfs}{U}{rsfs}{m}{n}
\DeclareSymbolFontAlphabet{\mathscr}{rsfs}

\newcommand{\thmheader}[1]{{\upshape (#1.)}}

\newcommand{\aff}{{\rmcmd{aff}}}

\newcommand{\lin}{\rmcmd{lin}}

\newcommand{\bL}{\mathbb{L}}

\newcommand{\card}[1]{\left\lvert #1 \right\rvert}

\newcommand{\conv}{\rmcmd{conv}}
\newcommand{\dotvar}{\,\cdot\,}
\newcommand{\floor}[1]{\lfloor {#1} \rfloor}

\newcommand{\integer}{\mathbb{Z}}
\newcommand{\Z}{\mathbb{Z}}
\newcommand{\R}{\mathbb{R}}
\newcommand{\Q}{\mathbb{Q}}

\newcommand{\intr}{\rmcmd{int}}

\newcommand{\N}{\mathbb{N}}
\newcommand{\rational}{\mathbb{Q}}

\newcommand{\rmcmd}[1]{\mathop{\mathrm{#1}}\nolimits}
\newcommand{\setcond}[2]{\left\{#1\,:\,#2\right\}}

\newcommand{\sprod}[2]{\left\langle#1,#2\right\rangle}

\newcommand{\term}[1]{\emph{#1}}
\newcommand{\vol}{\rmcmd{vol}}

\newcommand{\bM}{\mathbb{M}}

\newcommand{\cT}{\mathcal{T}}

\newcommand{\eps}{\varepsilon}

\newcommand{\myeqref}[1]{(\ref{#1})}

\newtheoremstyle{mythmstyle}
	{\topsep}
	{\topsep}
	{\itshape}
	{}
	{\scshape}
	{.}
	{3pt}
	{}
\theoremstyle{mythmstyle}

\newtheorem{nn}{}[section]
\newtheorem{corollary}[nn]{Corollary}
\newtheorem{example}[nn]{Example}
\newtheorem{remark}[nn]{Remark}
\newtheorem{lemma}[nn]{Lemma}
\newtheorem*{lemma*}{Lemma}
\newtheorem{proposition}[nn]{Proposition}
\newtheorem{theorem}[nn]{Theorem}

\newtheorem{problem}[nn]{Problem}

\usepackage{graphicx,xcolor}
\definecolor{update}{HTML}{FFFF00}
\definecolor{oldversion}{HTML}{990099}
\definecolor{commentB}{HTML}{66FF99}
\definecolor{commentG}{HTML}{33FFFF}

\begin{document}

\maketitle

\begin{abstract}
	Two sets in $\R^d$ are called homometric if they have the same covariogram, where the covariogram of a finite subset $K$ of $\R^d$ is the function associating to each $u \in \R^d$ the cardinality of $K \cap (K+u)$. Understanding the structure of homometric sets is important for a number of areas of mathematics and applications. 
	
	If two sets are homometric but do not coincide up to translations and point reflections, we call them nontrivially homometric. We study nontrivially homometric pairs of lattice-convex sets, where a set $K$ is called lattice-convex with respect to a lattice $\bM \subseteq \R^d$ if $K$ is the intersection of $\bM$ and a convex subset of $\R^d$. This line of research was initiated in 2005 by Daurat, G\'erard and Nivat and, independently, by Gardner, Gronchi and Zong. 
	
	All pairs of nontrivially homometric lattice-convex sets that have been known so far can essentially be written as direct sums $S \oplus T$ and $S \oplus (-T)$, where $T$ is lattice-convex, the underlying lattice~$\bM$ is the direct sum of $T$ and some sublattice $\bL$, and $S$ is a subset of $\bL$. We study pairs of nontrivially homometric lattice-convex sets assuming this particular form and establish a necessary and a sufficient condition for the lattice-convexity of $S \oplus T$. This allows us to explicitly describe all nontrivially homometric pairs in dimension two, under the above assumption, and to construct examples of nontrivially homometric pairs of lattice-convex sets for each $d \ge 3$.
\end{abstract}

\newtheoremstyle{itsemicolon}{}{}{\mdseries\rmfamily}{}{\itshape}{:}{ }{}
\newtheoremstyle{itdot}{}{}{\mdseries\rmfamily}{}{\itshape}{.}{ }{}
\theoremstyle{itdot}

\newtheorem*{msc*}{2010 Mathematics Subject Classification} 
\begin{msc*}
  Primary: 52C07; Secondary: 05B10, 52B20, 52C05, 78A45
\end{msc*}

\newtheorem*{keywords*}{Key words and phrases}
\begin{keywords*}
Covariogram, covariogram problem, crystallography, diffraction, discrete tomography, direct sum, homometric sets, lattice-convex set, quasicrystal,  X-ray.
\end{keywords*}

\section{Introduction}\label{sec:intro}

We assume acquaintance with basic concepts from convexity theory, the theory of polyhedra, and the geometry of numbers. For an extensive account on the background see, for example, \cite{MR1311028}, \cite{MR1940576}, \cite{MR2335496}, and \cite{Schneider2014}.

Throughout the text, let $d\in\N$, where $\N= \{1,2,3,\ldots\}$. For $X, Y \subseteq \R^d$, we define the (\emph{Minkowski}) \emph{sum}~$X + Y := \setcond{x + y }{x \in X, \ y \in Y}$, the set $-X := \setcond{-x}{x \in X}$, and the set~$X-Y:=X+(-Y)$. For $A \subseteq \R$, $X, Y \subseteq \R^d$, $a \in \R$, and $u,v \in \R^d$ we let~$A X := X A:= \setcond{a x }{a \in A, \ x \in X}$, 
$a X := X a :=\{a\} X = X\{a\}$, 
$u \pm Y := \{u\} \pm Y$, and $X \pm v := X \pm \{v\}$.  If~$k$ is the maximal possible number of affinely independent points in $X\subseteq\R^{d}$, the dimension $\dim(X)$ of $X$ is $k-1$.  Further, let $\sprod{\dotvar}{\dotvar}$ denote the standard scalar product of $\R^d$. 

\paragraph{Covariogram problems.}
For a finite subset $K$ of~$\R^d$, the \emph{covariogram} of $K$ is the function $g_{K}$ on $\R^d$ defined by
\[
	g_K(u) := \card{ K \cap (K+u)},
\]
where $\card{\dotvar}$ denotes the cardinality.
The covariogram occurs in the theory of \term{quasicrystals} and is closely related to the so-called \emph{diffraction data} of $K$ (see~\cite{quasicrystals-primer}, \cite{BaakeGrimm}, \cite{Lag00}, and \cite{Moo00}). The information provided by $g_K$ can be expressed in several equivalent algebraic, probabilistic, Fourier-analytic and tomographic terms; see \cite{RoSey82}, \cite{Daurat-Gerard-Nivat-2005} and \cite{MR2160045}. Furthermore, $g_K$ is related to the data provided by the so-called \emph{X-rays}, as explained in \cite[Theorem~4.1]{MR2160045}; see also \cite{MR1376547}. 

\newcommand{\cL}{\mathcal{L}}

The tomographic problems of reconstruction (of some geometric features) of an unknown set~$K$ from the knowledge of $g_K$ have attracted experts from different areas of mathematics and its applications; see the literature cited in~\cite{Lemke-Skiena-Smith-2003}. We call such types of problems \emph{covariogram problems}.  Note that each translation $K+t$ of $K$ by a vector $t \in \R^d$ and also each point reflection~$2 c - K$ of~$K$ with respect to a point $c \in \R^d$ has the same covariogram as $K$. This means, when no a~priori knowledge on $K$ is available, at best, $K$ can be determined by $g_K$ up to translations and point reflections. We call two finite sets $K, L$ \emph{homometric} if $g_K=g_L$, \emph{trivially homometric} if $K$ and~$L$ coincide up to translations and point reflections, and \emph{nontrivially homometric} if $K$ and~$L$ are homometric but not trivially homometric. See~\cite[Section~4]{BaakeGrimm} for the impact of nontrivial homometric pairs of lattice sets in the realm of quasicrystals.

If $K$ is centrally symmetric, that is, if $K$ is invariant under some point reflection, then~$g_K$ determines $K$ up to translations without any a~priori knowledge on $K$, that is, within the whole family of finite subsets of $\R^d$; see \cite{Averkov-detecting-cen-sym-2009}. Otherwise, however, some a~priori information on~$K$ is usually assumed. That means, the problem is to reconstruct an unknown $K$ using $g_K$ and relying on the fact that $K$ belongs to a family $\cL$, where $\cL$  represents the a~priori information on~$K$. 
If every~$L \in \cL$ homometric to $K$ is necessarily trivially homometric to $K$, then we say that $K$ is \emph{uniquely determined} by $g_K$ within $\cL$ (up to translations and point reflections).  There are choices of $\cL$, for which the unique determination of some sets $K \in \cL$ within $\cL$ is not possible. In such cases, a more appropriate aim is to describe, as explicitly as possible, the family of all sets $L \in \cL$ homometric to a given set $K \in \cL$. 

Covariogram problems can be introduced analogously in the `continuous' setting. For a compact set $K\subseteq\R^d$ with nonempty interior, $g_K(u)$ is defined to be the volume of $K \cap (K+u)$, for each~$u \in \R^d$. It is known that if $d=2$ and $K$ is a $d$-dimensional compact convex  set, then $K$ is uniquely determined by $g_K$ within the family of all $d$-dimensional compact convex sets, up to translations and point reflections; see~\cite{Averkov-Bianchi-2009}. The same holds true if $d=3$ and $K$ is a $d$-dimensional polytope; see~\cite{MR2493181}. So, in the continuous setting,  the convexity assumption on $K$ has led to a number of positive results on the reconstruction of~$K$ from the covariogram.

In 2005 the authors of \cite{Daurat-Gerard-Nivat-2005} and \cite{MR2160045} independently initiated the study of the retrieval of~$K$ from $g_K$ in the case that $K$ is a finite subset of $\R^d$ satisfying a condition analogous to convexity. As usual, a subset $\bM$ of $\R^d$ is called a lattice of rank $i \ge 0$ if, for some linearly independent vectors~$b_1,\ldots,b_i\in\R^d$, we have $\bM = \Z b_1 +  \cdots + \Z b_i$; in this case, $b_1,\ldots,b_i\in\R^d$ is called a \emph{basis} of the lattice $\bM$. From now on, let $\bM$ be a lattice of rank $d$ in $\R^d$. A subset $K$ of $\bM$ is called \emph{$\bM$-convex} if $K = C \cap \bM$ for some convex subset~$C$ of $\R^d$. In this definition, the choice of $\bM$ is not important; we could just fix $\bM=\Z^d$, but for technical reasons (which will be explained on page~\pageref{page:no-fixed-lattice-M}) it will be more convenient to consider an arbitrary lattice $\bM$ of rank $d$. Whenever the choice of $\bM$ is clear, we also say \emph{lattice-convex} rather than $\bM$-convex. The study of covariogram problems within the family of finite $\bM$-convex sets is the main motivating point for this manuscript. The covariogram problem is highly nontrivial for each dimension $d \ge 2$. For $d=1$, each $\bM$-convex set is centrally symmetric, so the problem obviously has a positive solution.

It seems that generally it is hard to transfer techniques developed for the covariogram problem within the family of convex bodies to the family of $\bM$-convex sets. One such attempt was made in our previous publication \cite{AveL12} for two-dimensional $\bM$-convex sets $K$. In \cite{AveL12} it was shown that if~$K$ samples $\conv(K)$ well enough, that is, if~$K$ is close enough to a compact convex set in a certain sense, then the reconstruction from~$g_K$ is similar to the reconstruction in the case of compact convex sets. But in general, there is a significant difference between the covariogram problem for the family of compact convex sets and its discrete analogue, the family of $\bM$-convex sets. This was already shown in \cite{Daurat-Gerard-Nivat-2005} and \cite{MR2160045}. Both these sources provide examples of pairs of \emph{nontrivially} homometric $\bM$-convex sets in dimension two, which shows that the positive result for planar convex bodies from~\cite{Averkov-Bianchi-2009} cannot be carried over to the discrete setting. Thus, the covariogram problem for $\bM$-convex sets appears to be intricate, as the properties of nontrivial pairs of homometric $\bM$-convex sets are not yet well understood. The aim of this manuscript is to provide tools and results which allow to gain more insight, at least under additional assumptions.

\paragraph{Nontrivially homometric pairs arising from the direct-sum operation.}
If, for $X,Y\subseteq \R^{d}$, each $z\in X+Y$ has a unique representation $z=x+y$ with $x\in X$ and $y\in Y$, we say that the sum~$X+Y$ is \emph{direct}; in this case we write $X\oplus Y$ for $X+Y$.
Direct sums provide an `easy template' to generate pairs of (nontrivially) homometric sets, as the following proposition shows. 

\begin{proposition}[Direct sums and homometry]
	\label{prp:dir:sum:hom}
	Let~$S$ and $T$ be finite subsets of $\R^d$ such that the sum of $S$ and $T$ is direct. Then the following statements hold:
		\begin{enumerate}[(a)]
			\item \label{S+(-T):too} The sum of $S$ and $-T$ is also direct.
			\item \label{S+T,S+(-T):hom:metr} $S \oplus T$ and $S \oplus (-T)$ are homometric. 
			\item \label{S+T,S+(-T):nontr:hom:metr} $S \oplus T$ and $S \oplus (-T)$ are nontrivially homometric if and only if both $S$ and $T$ are not centrally symmetric. 
		\end{enumerate}
\end{proposition}

The assertions of this proposition are rather basic and seem to be well known. Since we are not aware of proofs in the literature, we provide them as a service to the reader in Section~\ref{sect-proofs-all-dim}.

Though Proposition~\ref{prp:dir:sum:hom} makes it very easy to give examples of nontrivially homometric pairs of general finite sets, constructing pairs of  nontrivially homometric \emph{lattice-convex} sets is much more challenging. In our previous publication \cite{AveL12} we found the following infinite family of pairs $K, L$ of nontrivially homometric $\bM$-convex sets in dimension two:

\begin{example} \label{AveL12:example}
  Let~$d=2$ and $k\in\N$. Let~$\bM:=\Z^2$ and 
	\begin{align*}
\bL&:=\Z b_1 + \Z b_2, \ \text{where} \ b_1  = (k+1,-1), \ b_2 = (k,1), \\ 
T&:= \{0,\ldots,k\} \times \{0\} \cup \{0,\ldots,k-1\} \times \{1\}. 
	\end{align*}
Then $\bM=\bL\oplus T$; see Figures~\ref{fig:main:A}\,(a), (b), and (c). Choose $S$ to be any finite $\bL$-convex set such that $\conv(S)$ is a polygon and each edge of $\conv(S)$ is parallel to $b_1$, $b_2$, or $b_2-b_1=(-1,2)$. Then $K=S \oplus T$ and $L=S \oplus (-T)$ is a pair of homometric $\bM$-convex sets. If $S$ is not centrally symmetric, the pair is nontrivially homometric; see Figures~\ref{fig:main:A}\,(d), (e), and (f). 
\end{example}

\begin{figure}
  \begin{tabular}{p{4cm}p{5cm}p{5cm}}
    (a) $T$ \par\bigskip\qquad\includegraphics{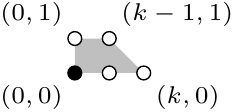} &%
    (b) Generators of $\bL$\par\bigskip\includegraphics{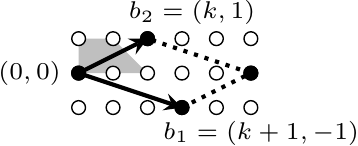} &
    (c) $\bM=\bL\oplus T$\par\medskip\includegraphics{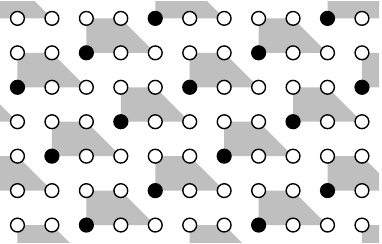}
    \\[1ex]%
    (d) $K=S\oplus T$\par\qquad\includegraphics{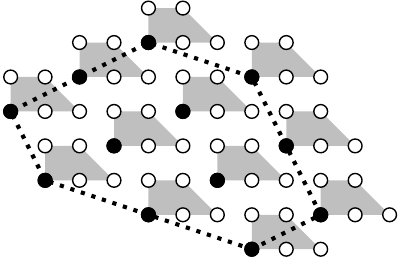}  &%
    (e) $L=S\oplus (-T)$\par\bigskip\qquad\includegraphics{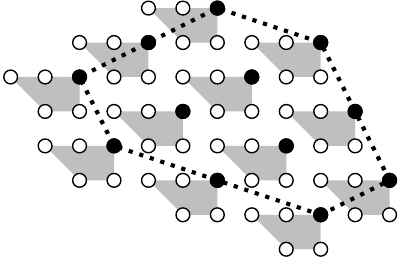}  &%
    (f)\par\bigskip\qquad\includegraphics{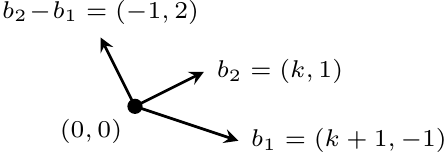}
  \end{tabular}
  \caption{Illustration to Example~\ref{AveL12:example} in the case $k=2$. The black dots are points of $S$ and $\bL$. The dotted lines in (d) and (e) represent the boundary of the convex hull of $S$. The gray regions are the convex hulls of translations of $T$ and $-T$. In (d) and (e), the unions of the white and black dots are the sets $K$ and $L$.}
\label{fig:main:A}
\end{figure}

Up to a change of coordinates in $\Z^{2}$ and up to translations of $K$ and $L$, the nontrivially homometric pairs $K, L$ from \cite{MR2160045} and \cite{Daurat-Gerard-Nivat-2005} are members of the family presented in Example~\ref{AveL12:example}. That is,  Example~\ref{AveL12:example} contains essentially all nontrivially homometric pairs of lattice-convex sets that have been known so far for $d=2$. To the best of our knowledge, for dimension $d\geq 3$, nontrivially homometric pairs of lattice-convex sets have not yet been studied. The primary aim of this manuscript is an extensive study of such pairs for an arbitrary dimension $d \ge 2$. Currently, it seems hard to carry out the study in full generality.  Therefore, we impose more structure that is similar to the structure in Example~\ref{AveL12:example}, as it is detailed below.

By $\cT^d$ we denote the set of all triples $(\bM,\bL,T)$ such that $\bM = \bL \oplus T$ holds, $\bM$ and $\bL$ are lattices of rank $d$ with $\bL \subseteq \bM$, and $T$ is a nonempty finite $\bM$-convex set. We call such a triple $(\bM,\bL,T)$ a \emph{tiling} of $\bM$ by translations of $T$ with vectors of $\bL$. The set $T$ in $(\bM,\bL,T) \in \cT^d$ is called a \emph{tile}. One can use the following procedure to generate some of the elements of $\cT^d$: Fix $\bM := \Z^d$, choose~$d$ linearly independent vectors $b_1,\ldots, b_d \in \Z^d$ and a vector $v \in \R^d$. Define $\bL$ to be the lattice generated by $b_1,\ldots,b_d$ and $T := \bM \cap \bigl( v + (0,1] b_1 + \cdots + (0,1]  b_d\bigr)$. Then~$(\bM,\bL,T) \in \cT^d$; see also Figure~\ref{fig:exa-Dirichlet}.  We emphasize that we restrict ourselves to tilings with a nonempty finite and $\bM$-convex tile $T$. The presence of two lattices $\bL$ and $\bM$ in our considerations explains why we do not want to fix $\bM=\Z^d$ throughout.\label{page:no-fixed-lattice-M} Indeed, equally well one could also fix $\bL=\Z^d$ and, for some of our arguments, such a choice will turn out to be more convenient.

For each $d \ge 2$ we want to describe nontrivially homometric pairs of $\bM$-convex sets which are generated from tilings $(\bM,\bL,T) \in \cT^d$ and have the form $S \oplus T$, $S \oplus (-T)$ with $S \subseteq \bL$.

\begin{figure}
  \begin{tabular}{p{4cm}p{4cm}}
  (a)\par\quad\includegraphics{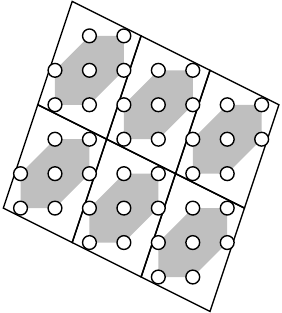}&
  (b)\par\quad\includegraphics{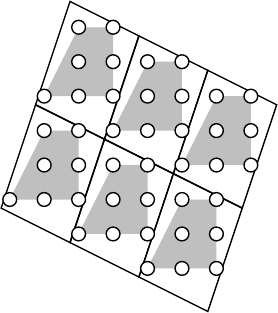}
  \end{tabular}
\caption{Examples of tilings $(\Z^{d},\bL,T) \in \cT^d$ generated by different translations of $(0,1] b_1 + \cdots + (0,1] b_d$. Here we have $d=2$ and $b_1=(2,-1)$ and $b_2=(1,3)$.}
   \label{fig:exa-Dirichlet}
\end{figure}

\paragraph{New contributions.} To formulate our main results, we need the notion of the width of a set. So, we introduce the \emph{support function} $h(K,\cdot)$ and the \emph{width function} $w(K,\cdot)$ of $K\subseteq\R^d$ as functions on $\R^{d}$ defined by 
\begin{align*}
        h(K,u)&:=\sup_{x\in K}\sprod{u}{x},&
  w(K,u) & :=  h(K,u) + h(K,-u) = \sup_{x \in K} \sprod{u}{x} - \inf_{x \in K} \sprod{u}{x}.
\end{align*}
Let~$\bL$ be the lattice of rank $d$ and consider its \emph{dual lattice} $\bL^\ast := \setcond{y \in \R^d}{ \text{$\sprod{y}{x} \in \Z$ for each $x \in \bL$}}$. The \emph{lattice width} of a set $K\subseteq\R^d$ with respect to  $\bL$ is defined by 
\[
  w(K,\bL):=\inf\setcond{w(K,u)}{u\in\bL^{*}\setminus\{o\}}.
\]
We also define the set 
\begin{align}\label{eq:def:WKL}
	W(K,\bL) &= \setcond{u \in \bL^\ast \setminus \{o\}}{w(K,u)<1},
\end{align}
that is, the set of all directions in which $K$ is `thin' relative to $\bL$. We observe that $W(K,\bL)$ contains, up to rescaling, all vectors $u \in \R^d \setminus \{o\}$ with the property that the family of hyperplanes orthogonal to $u$ and passing through points of $\conv(K) + \bL$ does \emph{not} cover the whole space $\R^d$; see Figure~\ref{fig:width:AveL12:strips} for an example.\footnote{This geometric interpretation is a straightforward consequence of the definition of $W(K,\bL)$; we only use it for the purpose of visualization.}

\begin{figure}
  \begin{tabular}{p{3.5cm}p{3.5cm}p{3.5cm}p{3.5cm}}
    (a)\par\smallskip\qquad\includegraphics{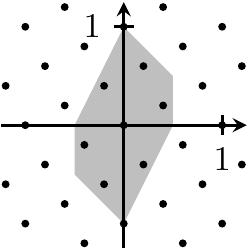}&
    (b)\par\smallskip\includegraphics{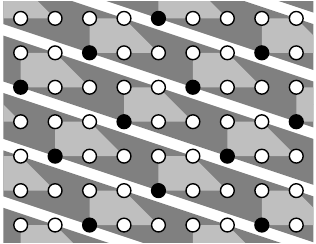}&
    (c)\par\smallskip\includegraphics{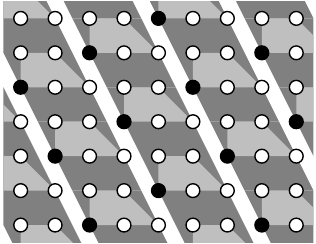}&
    (d)\par\smallskip\includegraphics{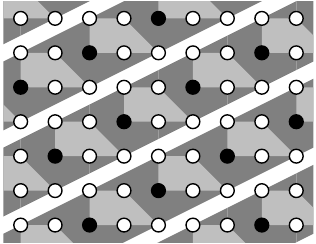}
   \end{tabular}
   \caption{The set $W(T,\bL)$ and its geometric meaning for the tiling $(\bM,\bL,T) \in \cT^2$ in Example~\ref{AveL12:example} with $k=2$. In (a) the gray area is $\setcond{u \in \R^d}{w(T,u) \le 1}$ and the dots are the elements of $\bL^\ast$; one can see that $\card{W(T,\bL)}=6$. Each pair $u,-u$ of vectors in $W(T,\bL)$ defines a family of strips orthogonal to $u$ which cover $\conv(T) + \bL$ but do not cover $\R^2$; see (b), (c), and (d).}
   \label{fig:width:AveL12:strips}
\end{figure}

For $(\bM,\bL,T) \in \cT^d$, we establish the following results: 

\begin{enumerate}[I.]
\item \emph{A necessary and a sufficient condition for $\bM$-convexity of $S\oplus T$.}  We give a sufficient condition for the $\bM$-convexity of $S \oplus T$, when $S \subseteq \bL$ is $d$-dimensional. This condition is formulated in terms of~$W(T,\bL)$. We also derive a similar necessary condition. Both conditions are presented in Theorem~\ref{thm:ABC}, which is the main tool for proving all subsequent results.

\item \emph{New examples of nontrivially homometric lattice-convex sets.} \label{new:examples:item} For each $d \ge 2$ we give tilings $(\bM,\bL,T) \in \cT^d$ such that $W(T,\bL)$ contains $d+1$ distinct vectors that linearly span~$\R^d$. These tilings generate many examples of nontrivially homometric pairs of lattice-convex sets of the form $S \oplus T$ and $S \oplus (-T)$; see Corollary~\ref{nontrivial-S-cor} and Example~\ref{AveL12:example:generalized}. We emphasize that so far, for every $d \ge 3$ no pairs of nontrivially homometric lattice-convex sets (apart from those which are lifted from dimension two by taking Cartesian products) have been known.

\item \emph{Finiteness result for $W(T,\bL)$.} Result \ref{new:examples:item} motivates the study of the case that $W(T,\bL)$ contains a basis of $\R^d$. Assume that, for fixed $\bL$,  we vary $T$ and $\bM$, preserving the condition $(\bM,\bL,T) \in \cT^d$. We show that, up to a change of the basis of $\bL^\ast$, there are only finitely many sets $W(T,\bL)$ such that $T$ is $d$-dimensional and $W(T,\bL)$ contains a basis of $\R^d$.

\item \emph{Explicit description in the plane.} We describe
(up to a change of coordinates and up to translations) all nontrivially homometric pairs of $\bM$-convex sets of the form $S \oplus T$, $S \oplus (-T)$ in the plane. Theorem~\ref{thm:complete-classification} shows that Example~\ref{AveL12:example} already contains \emph{all} such pairs.
\end{enumerate}

\paragraph{Structure of the paper.} In Section~\ref{sec:results} we state our main results. The proofs for arbitrary dimension $d$ are given in Section~\ref{sect-proofs-all-dim}. Section~\ref{sec:proof:classification} will give the proof of Theorem~\ref{thm:complete-classification}, which also includes a computer enumeration that we perform with the computer algebra system Magma \cite{MR1484478}. Magma directly supports various computations for rational polytopes and lattice points that we need for our enumeration (in particular, the computation of the lattice width).\footnote{Note that one could also use Polymake \cite{MR2721537} or, with somewhat higher amount of work, a standard general-purpose programming language.}
Section~\ref{sec:examples} presents constructions of tilings which can be used together with the main theorems in order to generate nontrivially homometric pairs of lattice-convex sets. As a complement to this,  Section~\ref{sec:examples} also contains examples that illustrate the impact of the different assumptions that we make in our results. Finally, in Section~\ref{sec:magma-code} we present the Magma code that we used to accomplish different computational tasks, in particular the computer assisted proof of the explicit description for dimension two.  We close this section by collecting some further notions and standard terminology.

In $\R^d$, $o$ denotes the origin and $e_1,\ldots,e_d$ the standard basis. The set $D(X):=X-X$ is called the \emph{difference set} of $X\subseteq\R^d$. The greatest common divisor of the entries of $v\in\Z^d$ is denoted by $\gcd(v)$. The volume, i.\,e., the Lebesgue measure on $\R^d$, is denoted by $\vol$.  We use $\intr(K)$, $\conv(K)$, $\lin(K)$ and $\aff(K)$ to denote the interior, the convex hull, the linear hull, and the affine hull of $K\subseteq \R^d$, respectively. For $x, y \in \R^d$ we let $[x,y]:=\conv(\{x,y\})$.
We use common terminology for convex sets, polytopes and polyhedra such as face, facet, vertex, (outer) normal vector to a facet, rational polyhedron and integral  polyhedron. Using the support function, we define $F(K,u):=\setcond{x\in K}{\sprod{u}{x}=h(K,u)}$, where $u \in \R^d$. For a polytope $P$ in $\R^d$, the set $F(P,u)$ is a face of~$P$. The \emph{polar} of $K$ is the closed convex set $K^\circ := \setcond{x\in\R^d}{\text{$\sprod{x}{y}\leq 1$ for each $y\in K$}}$. 

Let~$\bL$ be a lattice. A nonzero vector~$x$ of $\bL$ is called a \emph{primitive vector} of $\bL$ if $o$ and $x$ are the only points of $\bL$ on the line segment $[o,x]$. Assume that $\bL\subseteq\R^d$ has rank $d$ and that $b_1,\ldots,b_d\in\R^d$ is a \emph{basis} of $\bL$. Then  $\sum_{j=1}^i(0,1]b_{j}$ is called the \emph{Dirichlet cell} with respect to the basis $b_1,\ldots,b_d$ and the \emph{determinant} of $\bL$ is $\det(\bL) := \vol \bigl( (0,1] b_1 + \cdots + (0,1] b_d\bigr)$. The number $\det(\bL)$ is independent of the choice of the basis. Further, the dual lattice $\bL^\ast$ is itself a lattice of rank $d$. Clearly, $(\Z^{d})^{*}=\Z^{d}$. The \emph{dual basis} of a basis $b_1,\ldots,b_d$ of $\bL$ is the uniquely determined basis $b_1^\ast,\ldots,b_d^\ast$ satisfying $\sprod{b_i}{b_i^\ast} = \delta_{i,j}$ for all $i, j \in \{1,\ldots,d\}$, where $\delta_{i,j}$ is the Kronecker delta. If $b_1,\ldots,b_d$ is a basis of the lattice $\bL$, then $b_1^\ast,\ldots,b_d^\ast$ is a basis of $\bL^\ast$. The latter also implies $\det(\bL^\ast)=1/\det(\bL)$.

A matrix $U \in \Z^{d \times d}$ is called \emph{unimodular} if its determinant is $1$ or $-1$. A mapping on $\R^d$ is called a \emph{unimodular transformation} if it can be written as $x\mapsto Ux$ for some unimodular matrix $U$; it is called a \emph{affine unimodular transformation} if it can be written as $x\mapsto Ux+z$ for some unimodular matrix $U$ and some $z\in\Z^d$. It is well known that (affine) unimodular transformations are exactly those (affine) linear mappings on $\R^d$ that are $\Z^d$-preserving, that is, that map $\Z^d$ bijectively to $\Z^d$.

\section{Main results}\label{sec:results}

\begin{theorem}\label{thm:ABC} \thmheader{A sufficient and a necessary condition for the $\bM$-convexity of $S \oplus T$}
  Let~$(\bM, \bL,T) \in \cT^d$ and let $S$ be a finite $d$-dimensional subset of $\bL$. Then the implications \myeqref{item:ABC:A}\,$\Rightarrow$\,\myeqref{item:ABC:B}\,$\Rightarrow$\,\myeqref{item:ABC:C} hold, where \myeqref{item:ABC:A}, \myeqref{item:ABC:B} and \myeqref{item:ABC:C} are the following conditions:
  \begin{enumerate}[(a)]
    \item\label{item:ABC:A} $S$ is $\bL$-convex and each facet of the polytope $\conv(S)$ has a normal vector in $W(T,\bL)$.
    \item\label{item:ABC:B} $S \oplus T$ is $\bM$-convex.
    \item\label{item:ABC:C} $S$ is $\bL$-convex and each facet $F$ of $\conv(S)$ with $\aff(F)\subseteq F + \bL$ has a normal vector in~$W(T,\bL)$.
  \end{enumerate}
\end{theorem}

We illustrate implication \myeqref{item:ABC:A}\,$\Rightarrow$\,\myeqref{item:ABC:B} by Example~\ref{AveL12:example}; see also Figures~\ref{fig:main:A} and \ref{fig:width:AveL12:strips} for $k=2$: The validity of~\myeqref{item:ABC:A} means that the edges of the polygon $\conv(S)$ are parallel to the strips introduced in Figure~\ref{fig:width:AveL12:strips}. Thus, Theorem~\ref{thm:ABC} confirms that the sets $S \oplus T$ in Example~\ref{AveL12:example} are $\bM$-convex. 

The condition $(\bM,\bL,T) \in \cT^d$ and the definition of $W(T,\bL)$ are invariant with respect to replacing~$T$ by~$-T$. So, if \myeqref{item:ABC:A} holds for $T$, it also holds for $-T$ and we conclude that both sets in the homometric pair $S \oplus T$, and $S \oplus (-T)$ are $\bM$-convex. Furthermore, \myeqref{item:ABC:A} is invariant with respect to replacing $S$ with $S_k:=\conv(k S) \cap \bL$, where $k \in \N$. Thus, whenever we can use \myeqref{item:ABC:A}\,$\Rightarrow$\,\myeqref{item:ABC:B} to find an $\bM$-convex set of the form $S \oplus T$, we get infinitely many such sets $S_k \oplus T$ with $k \in \N$. Also note that  $W(T,\bL)$ can be computed algorithmically when, say, $\bL=\Z^d$ and  $T$ is a finite, $d$-dimensional subset of $\rational^d$ (see also Section~\ref{sec:magma-code}). This paves the way to a computer-assisted search for interesting pairs.

In view of Proposition~\ref{prp:dir:sum:hom}, for given $(\bM, \bL, T) \in \cT^d$, the implication \myeqref{item:ABC:A}\,$\Rightarrow$\,\myeqref{item:ABC:B} can be used to search for sets $S$ with $S \oplus T$, $S \oplus (-T)$  being \emph{nontrivially} homometric: 

\begin{corollary} \label{nontrivial-S-cor} \thmheader{A sufficient condition for the existence of $S$ such that $S \oplus T$ and $S \oplus (-T)$ are nontrivially homometric and $\bM$-convex}
Let~$(\bM, \bL, T) \in \cT^d$ with noncentrally symmetric $T$. Let~$W(T,\bL)$ contain linearly independent vectors $u_1,\ldots,u_d$ and a vector $u_{d+1}$ that is not parallel to any of the vectors $u_1,\ldots,u_d$. Then there exists $S\subseteq\bL$ being noncentrally symmetric, finite, $d$-dimensional, and $\bL$-convex such that each facet of $\conv(S)$ has an outer normal vector in $\{\pm u_1,\ldots,\pm u_{d+1}\}$. For each such $S$, the sets $S \oplus T$ and $S \oplus (-T)$ form a pair of nontrivially homometric $\bM$-convex sets. 
\end{corollary}

For this corollary, too, the existence of one set $S$ implies the existence of infinitely many sets $S_k := \conv(k S) \cap \bL$ with $k \in \N$ that satisfy the same assertion.  In Section~\ref{sec:examples} we construct, for each~$d \ge 2$, tilings $(\bM,\bL,T) \in \cT^d$ which satisfy the assumptions of Corollary~\ref{nontrivial-S-cor}.

Corollary~\ref{nontrivial-S-cor} motivates the study of sets $T$ and $\bL$ with $W(T,\bL)$ containing $d$ linearly independent vectors. The following theorem asserts that whenever $T$ is finite and $d$-dimensional, there are essentially finitely many such sets $W(T,\bL)$. This limits the search space for nontrivially homometric pairs considerably and implies that, loosely speaking, such pairs are `rare'.

\begin{theorem} \thmheader{Finitely many shapes of $W(T,\bL)$} \label{finiteness-thm}
  Let~$(\bM,\bL,T) \in \cT^d$, let $T$ be $d$-dimensional, and let $W(T,\bL)$ contain $d$ linearly independent vectors. Then the following statements hold:
  \begin{enumerate}[(a)]
  \item\label{item:finiteness-thm:III} There exists a basis $b_1^\ast,\ldots,b_d^\ast$ of the lattice $\bL^\ast$ and $k\in\N$ with $k\leq  2 (3/2)^{d-2} (d!)^2$ and $W(T,\bL) \subseteq \sum_{i=1}^{d}\{-k,\ldots,k\} b_i^\ast$.
    \item\label{item:finiteness-thm:I} $\card{W(T,\bL)} < 4^d$.
    \item\label{item:finiteness-thm:II} $\vol(\conv(W(T,\bL))) < \vol(D(T)^\circ)\leq 4^d \det(\bL^\ast) = 4^d/\det(\bL)$.
  \end{enumerate}
\end{theorem}

Theorem~\ref{finiteness-thm} shows that there are essentially finitely many sets $W(T,\bL)$ satisfying the assumptions of Theorem~\ref{finiteness-thm}. In fact, choosing appropriate coordinates (see also Proposition~\ref{change:prp}) we can assume that the basis $b_1^\ast,\ldots,b_d^\ast$ in~\myeqref{item:finiteness-thm:III} is the standard basis $e_1,\ldots,e_d$. Then $\bL^\ast=\Z^d$ and $W(T,\bL)$ is a subset of the finite set $\{-k,\ldots,k\}^d$ with $k$ as in~\myeqref{item:finiteness-thm:III}. Thus, all the finitely many possible `shapes' of $W(T,\bL)$ can be found in $\{-k,\ldots,k\}^d$, which is a set of cardinality $(2k+1)^d$. So there are at most~$2^{(2k+1)^d}$ essentially different shapes of $W(T,\bL)$. This bound could easily be improved, for example, by taking into account~\myeqref{item:finiteness-thm:I} and~\myeqref{item:finiteness-thm:II}, but we do not elaborate on this.

In Theorem~\ref{thm:ABC} for  $d=2$ the condition $\aff(F) \subseteq F + \bL$ clearly holds for each facet $F$ of $\conv(S)$. So, for $d=2$ we get \myeqref{item:ABC:A}\,$\Leftrightarrow$\,\myeqref{item:ABC:B} in Theorem~\ref{thm:ABC} and hence a characterization of the $\bM$-convexity of~$S \oplus T$. Based on this we give the following explicit description:

\begin{theorem}[Complete classification in dimension two]\label{thm:complete-classification}
  Let~$d=2$, $(\bM,\bL,T) \in \cT^d$, and $S \subseteq \bL$. Then the following statements are equivalent:
  \begin{enumerate}[(i)]
  \item\label{thm-item:complete-classification:ST-give-NontrHomPair} $S \oplus T$ and $S \oplus (-T)$ form a pair of nontrivially homometric $\bM$-convex sets. 
  \item\label{thm-item:complete-classification:T-has-width-1} There exist $k \in \N$, a basis $a_1,a_2$ of $\bM$, and a basis $b_1,b_2$ of $\bL$ such that the following conditions hold:
  \begin{enumerate}[(a)]
    \item $b_1= (k+1) a_1 - a_2$ and $b_2= k a_1 + a_2$.
    \item There exists $v \in \bM$ such that $T+v = \{0,\ldots,k\} a_1 \cup \bigl( \{0,\ldots,k-1\} a_1 +  a_2 \bigr)$.
    \item The set $S$ is noncentrally symmetric, finite, two-dimensional, $\bL$-convex, and every edge of the polygon $\conv(S)$ is parallel to $b_1$ or $b_2$ or $b_2-b_1$.
  \end{enumerate}
  \end{enumerate}
\end{theorem}
Parts (a) and (b) of Theorem~\ref{thm:complete-classification}\,\myeqref{thm-item:complete-classification:T-has-width-1} are illustrated by Figure~\ref{fig-shape-of-T}, part (c) is illustrated in Figures~\ref{fig:main:A}\,(d), (e), and (f).

  \begin{figure}[htb]
  \begin{center}
          \begin{tabular}{p{6cm}}
                \includegraphics{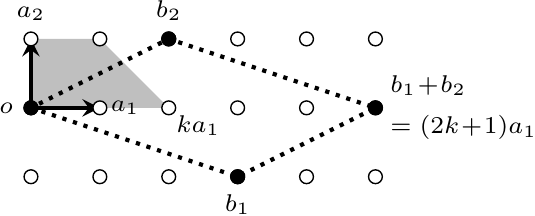}
              \end{tabular}
  \end{center}
  \caption{Illustration of Theorem~\ref{thm:complete-classification}\,\myeqref{thm-item:complete-classification:T-has-width-1}\,(a) and (b) with $k =2$. The gray region is $\conv(T+v)$.}
  \label{fig-shape-of-T}
  \end{figure}
  In view of the presented results, we formulate the following problems.

\begin{problem}
  We call pairs $K, L$ of homometric sets, which can (resp.\ cannot) be represented as $K= S \oplus T$, $L = S \oplus (-T)$, {geometrically constructible} (resp.\ {geometrically inconstructible}).  Do there exist geometrically inconstructible homometric pairs of lattice-convex sets? This question was raised in \cite[Question~2.4]{AveL12} for $d=2$. Note that, for this question, the lattice-convexity assumption on $K,L$ is crucial; without it, geometrically inconstructible pairs $K, L$ of homometric sets \emph{do} exist for every $d \ge 1$; see \cite{RoSey82}. 
\end{problem}

\begin{problem}
 We ask whether Theorem~\ref{thm:ABC} can be improved to a more precise description of sets~$S \subseteq \bL$ for which~$S \oplus T$ is $\bM$-convex. We also ask whether Theorem~\ref{thm:ABC} can be extended to sets $S$ of any dimension.
\end{problem}

\begin{problem}
 In the situation of Theorem~\ref{finiteness-thm} we ask for an explicit enumeration of all possible sets $W(T,\bL)$ for fixed dimensions, for example for $d=2$ and $d=3$.
\end{problem}

\section{Proofs of results for arbitrary dimension}\label{sect-proofs-all-dim}

We start with a proof of Proposition~\ref{prp:dir:sum:hom}. 

\begin{proof}[Proof of Proposition~\ref{prp:dir:sum:hom}]
	\myeqref{S+(-T):too}: This assertion was mentioned in \cite[p.~221]{AveL12} without proof. Consider an arbitrary $x \in S \oplus (-T)$ and any $s, s' \in S$ and $t,t' \in T$ with $x = s-t = s' - t'.$ We need to verify that $s=s'$ and $t=t'$. We have $x+t+t'=  s+t'\in S\oplus T$ and also $x+t+ t' =s' + t\in S\oplus T$. Since the sum of $S$ and $T$ is direct, we obtain $s= s'$ and $t=t'$. 
	
	\myeqref{S+T,S+(-T):hom:metr}: It is straightforward to check that for every finite $K \subseteq \R^d$ and every $u \in \R^d$, one has 
	\begin{equation*}
		g_K(u) = \card{ \setcond{(x_1,x_2)}{x_1,x_2 \in K, \ u=x_1-x_2}}.
	\end{equation*}
	Using the fact that the mapping $(s,t) \mapsto s +t$ is a bijection from $S \times T$ to $S \oplus T$, we obtain
	\begin{equation}
		\label{g:plus:expr}
		g_{S \oplus T}(u) = \card{ \setcond{(s_1,t_1,s_2,t_2)}{s_1,s_2 \in S, \ t_1, t_2 \in T, \ u=(s_1 + t_1) - (s_2 + t_2)} }.
	\end{equation}
	Analogously, for $S \oplus (-T)$ we have 
	\begin{equation}
		\label{g:minus:expr}
		g_{S \oplus (-T)}(u) = \card{ \setcond{(s_1,t_1,s_2,t_2)}{s_1,s_2 \in S, \ t_1, t_2 \in T, \ u=(s_1 - t_1) - (s_2 - t_2)} }.
	\end{equation}
	The set used in the right-hand side of \myeqref{g:plus:expr} is mapped bijectively to the set used in the right-hand side of \myeqref{g:minus:expr} via the bijection $(s_1,t_1,s_2,t_2) \mapsto (s_1,t_2,s_2,t_1)$. This shows $g_{S \oplus T}(u) = g_{S \oplus (-T)}(u)$. 
	
	\myeqref{S+T,S+(-T):nontr:hom:metr}: We show that $S \oplus T$ and $S \oplus (-T)$ are trivially homometric if and only if $S$ or $T$ is centrally symmetric. The sufficiency is straightforward. For the necessity assume that $S \oplus T$ and $S \oplus (-T)$ are trivially homometric, which means by the definition of trivial homometry that one of the following two cases occurs.
	
	\emph{Case~1: $S \oplus T$ and $S \oplus (-T)$ coincide up to a translation.} One has $S \oplus T = c + (S \oplus (-T))$ for some $c \in \R^d$. We verify the central symmetry of $T$ by showing  $T = c - T$ by induction on~$\card{T}$. For~$\card{T} \le 1$ there is nothing to show. Assume that $\card{T} \ge 2$ and that for sets of smaller cardinality the assertion is true. 
 	Taking the convex hull, we get $\conv(S) + \conv(T) = c + \conv(S) + \conv(-T)$. By the cancellation law for the Minkowski sum (see \cite[p.\ 48]{Schneider2014}), we obtain $\conv(T) = c + \conv(-T) = c - \conv(T)$. That is, $\conv(T)$ is centrally symmetric. In particular, since $\card{T} \ge 2$, there exist two distinct $t_1, t_2 \in T$ such that $t_1, t_2$ are vertices of $T$ and $t_1 = c - t_2$. For the set $T'  := T \setminus \{t_1, t_2\}$ one has $S \oplus T' = c + (S \oplus (-T'))$. Using the induction assumption, we conclude that $T' = c - T'$. The latter implies $T = c - T$.
	
	\emph{Case~2: $S \oplus T$ and $S \oplus (-T)$ coincide up to a reflection in a point.} One has $S \oplus T = c - (S \oplus (-T))$  for some $c \in \R^d$. Analogously to the previous case, one can show by induction on $\card{S}$ that $S = c -S$, that is, $S$ is centrally symmetric. 
\end{proof}

\begin{remark}
	Parts~\myeqref{S+T,S+(-T):hom:metr} and~\myeqref{S+T,S+(-T):nontr:hom:metr} of Proposition~\ref{prp:dir:sum:hom} can also be proved algebraically, using the group ring $\Z[\R^d]$, by employing tools from \cite[Theorem~2.1]{RoSey82}.
\end{remark}

The following statement will be used in the proof of Theorem~\ref{thm:ABC} and also later in Section~\ref{sec:proof:classification}.
\begin{lemma} \label{lem:S+T:convex=>S:convex}
	Let~$(\bM,\bL,T) \in \cT^d$, let $S \subseteq \bL$ and let $S \oplus T$ be $\bM$-convex. Then $S$ is $\bL$-convex. 
\end{lemma}
\begin{proof} We need to show $\bL\cap \conv(S)\subseteq S$. In view of $\bM=\bL\oplus T$, the latter is equivalent to $(\bL\cap \conv(S)) \oplus T\subseteq S\oplus T$. Since $S\oplus T$ is $\bM$-convex, the previous inclusion amounts to $(\bL\cap \conv(S)) \oplus T \subseteq \conv(S\oplus T)\cap \bM$, which is straightforward.
\end{proof}

For a $d$-dimensional finite $\bL$-convex set $S$ in $\R^d$ we define 
$U(S,\bL)$ to be the set of all primitive vectors $u$ of the lattice $\bL^\ast$ which are outer normals to facets of the polytope $\conv(S)$. 

\begin{proof}[Proof of Theorem~\ref{thm:ABC}] \myeqref{item:ABC:A}\,$\Rightarrow$\,\myeqref{item:ABC:B}:  Let \myeqref{item:ABC:A} hold, that is, the set $S$ is $\bL$-convex and $w(T,u)<1$ for each $u \in U(S,\bL)$. Since $S \subseteq \bL$ and $\bM = \bL \oplus T$, the sum of $S$ and $T$ is clearly direct. It remains to show that $S \oplus T$ is $\bM$-convex.
It suffices to show $\bM \cap \conv ( S \oplus T) \subseteq S \oplus T$. We have
 \begin{equation*}
  \bM \cap \conv( S \oplus T) = \bM \cap (\conv(S) + \conv(T)) = (\bL \oplus T) \cap (\conv(S) + \conv(T)).
 \end{equation*}
 Thus, each element of $\bM \cap \conv (S \oplus T)$ can be given as $s+t = s'+t'$ with $s \in \bL$, $t \in T$, $s' \in \conv(S)$, and $t'\in \conv(T)$. We show that, for $s,t, s', t'$ as above, one has $s \in S$. We argue by contradiction, so assume that $s \not\in S$. Since $S = \bL \cap \conv(S)$, we obtain $s \not\in \conv(S)$. Hence $S$ and~$s$ lie on different sides of the affine hull of some facet of $\conv(S)$, that is,  $\sprod{s}{u} > h(S,u)$ for some $u \in U(S,\bL)$. In view of $u \in \bL^\ast$ we have $\sprod{s}{u}, h(S,u) \in \integer$, so  $\sprod{s}{u} > h(S,u)$ can be improved to $\sprod{s}{u} \ge h(S,u)+1$. By assumption, $\sprod{u}{t'-t}\leq w(T,u) < 1$. Hence $\sprod{u}{s+t-t'} > h(S,u)$. On the other hand, $s + t-t' = s' \in \conv(S)$ and so $\sprod{u}{s+t-t'} \le h(S,u)$, a contradiction.

\myeqref{item:ABC:B}\,$\Rightarrow$\,\myeqref{item:ABC:C}:  Let~$S\oplus T$ be $\bM$-convex. Then $S$ is $\bL$-convex due to Lemma~\ref{lem:S+T:convex=>S:convex}. Consider an arbitrary facet $F$ of $\conv(S)$ satisfying $\aff(F) \subseteq F + \bL$. Let~$u\in U(S,\bL)$ be an outer normal of~$F$, that is, $F(\conv(S),u) =F$. We show $u \in W(T,\bL)$ by contradiction. So assume $u \notin W(T,\bL)$. Then $w(T,u) \ge 1$. Replacing $T$ with an appropriate translation of $T$ by a vector in $\bM$, we assume $o \in T$ and $w(T,u) = h(T,u)$. Replacing $S$ with an appropriate translation of $S$ by a vector in~$\bL$, we assume $o \in F(S,u)$. Choose any $x \in F(T,u)$. The set $P:= \conv(F \cup (F+x))$ is a prism with bases $F$ and $x+F$. We introduce the hyperplane $H:= \setcond{y \in \R^d}{\sprod{u}{y} = \floor{h(T,u)}}$, where $\floor{h(T,u)} \ge 1$, because $h(T,u) = w(T,u) \ge 1$.  The section $Q := P \cap H$ of the prism $P$ coincides with its base $F$, up to translations. 

Let us first show that $Q \cap \bL \ne \emptyset$. We have $H \cap \bL \ne \emptyset$ since $H$ is defined by a primitive vector~$u\in\bL^\ast$ and $\floor{h(T,u)}$ in the definition of $H$ is integer. Choose $a \in H \cap \bL$. The base $F$ of~$P$ and the section $Q$ of $P$ coincide up to translations. That is, $F = Q + v$ for some $v \in \R^d$. It follows that $a \in H = \aff(Q) = \aff(F) - v$ and so, $a+v \in \aff(F)$. Since $\aff(F) \subseteq F + \bL$, we get $a + v \in F + \bL$. Thus, $a +v \in F + b$ for some $b \in \bL$. Consequently, $a - b \in \bL$ and $a-b \in F - v = Q$. This shows $Q \cap \bL \ne \emptyset$. 

Choose any $z \in Q \cap \bL$. Since $z \in Q \subseteq P \subseteq \conv(S+T)$ and $z \in \bL \subseteq \bM$, we get $z \in \conv(S \oplus T) \cap \bM$. By \myeqref{item:ABC:B} we have $z \in S \oplus T$, so $z= s + t$ for some $s \in S$ and $t \in T$. The condition $o \in F(S,u)$ implies $\sprod{s}{u} \le 0$. Since $z \in Q \subseteq H$, we get $\sprod{z}{u} = \floor{h(T,u)} \ge 1$. Thus, on the one hand, $t=z-s \in \bL$ and, on the other hand, $\sprod{t}{u} = \sprod{z}{u} - \sprod{s}{u} \ge 1$ and hence $t \ne o$. It follows that $o$ and $t$ are distinct points of $\bL$, both belonging to $T$. This is a contradiction to the fact that the sum of $\bL$ and $T$ is direct. 
\end{proof}

\begin{remark}[Relation to the covering radius]
	The property $\aff(F) \subseteq F+ \bL$ in Theorem~\ref{thm:ABC}\,\myeqref{item:ABC:C} can be expressed using the well-known notion of the {covering radius} (also called {inhomogenious minimum}); see \cite[p.\ 381]{MR893813} and \cite[p.\ 579]{0659.52004}. To illustrate this, assume for simplicity that $o \in \aff(F)$, so that $\aff(F)$ is a linear space. In this case, the above property means that the covering radius $\mu$ of the $(d-1)$-dimensional polytope $F$ in the $(d-1)$-dimensional linear space $\aff(F)$ with respect to the lattice~$\aff(F) \cap \bL$ satisfies the inequality $\mu \le 1$.
\end{remark}

In the following considerations, we frequently switch to `more convenient' coordinates. The change of coordinates is motivated by the following proposition, which uses the notation $A(X):=\setcond{Ax}{x\in X}$ for a matrix $A \in \R^{d \times d}$ and a set $X\subseteq\R^d$ and the matrix $(A^{-1})^\top$ (the transposed of the inverse matrix of $A$). 

\begin{proposition} \label{change:prp}
	Let~$(\bM,\bL,T) \in \cT^d$. Let~$A \in \R^{d \times d}$ be a nonsingular matrix. Then the following relations hold: 
	\begin{align*}
		\bigl(A(\bM),A(\bL), A(T)\bigr) & \in \cT^d,
		\\ (A(\bM))^\ast & = (A^{-1})^\top(\bM^\ast),
		\\ (A(\bL))^\ast & = (A^{-1})^\top(\bL^\ast),
		\\ W\bigl(A(T),A(\bL)\bigr) & = (A^{-1})^\top\bigl( W(T,\bL) \bigr).
	\end{align*}
\end{proposition}
We omit the straightforward proof of Proposition~\ref{change:prp}, relying on basic properties of duality of lattices and polarity of sets. In view of Proposition~\ref{change:prp}, choosing an appropriate $A$ (resp. $(A^{-1})^\top$), we will be able to assume that $\bL$ or $\bM$ (resp.\ $\bL^\ast$ or $\bM^\ast$) is $\Z^d$. Furthermore, Proposition~\ref{change:prp} allows to keep track of the respective change of the set $W(T,\bL)$.

\begin{proof}[Proof of Corollary~\ref{nontrivial-S-cor}]
	After possibly changing coordinates in $\R^d$ we assume $b_i=e_i$ for every $i \in \{1,\ldots,d\}$ and hence $\bL=\bL^\ast = \Z^d$. Let~$b:=b_{d+1} \in \Z^d \setminus \{o\}$. Since $b$ is not parallel to any of the vectors $e_1,\ldots,e_d$, the faces $F(C,b)$ and $F(C,-b)$ of the cube $C:=[-1,1]^d$ are not facets. Fix $\eps\in\Q$ with $0< \eps < h(C,b)$. The polytope $P  := \setcond{x \in C}{\sprod{b}{x} \le h(C,b) - \eps }$  is rational, $d$-dimensional, and has a facet with outer normal $b$. Hence, $P$ is not centrally symmetric (because $F(P,b)$ is a facet of $P$ but $F(P,-b)$ is not) and each facet of $P$ has a normal vector in $\{e_1,\ldots,e_d,b\}$. Fix $k\in\N$ such that the polytope $k P$ is integral. Let~$S:= (kP) \cap \bL$. Since~$P$ is not centrally symmetric, $S$, too, is not centrally symmetric. Applying Theorem~\ref{thm:ABC}, we conclude that $S \oplus T$ and $S \oplus (-T)$ are $\bM$-convex. Since neither $S$ nor $T$ is centrally symmetric, by Proposition~\ref{prp:dir:sum:hom}\,\myeqref{S+T,S+(-T):nontr:hom:metr}, the sets $S \oplus T$ and~$S \oplus (-T)$ form a nontrivially homometric pair. 
\end{proof}

For the proof of Theorem~\ref{finiteness-thm} we recall some known results from the geometry of numbers.
The following theorem is contained in~\cite[Theorem~2 of \S10]{MR893813}.\footnote{In Theorem~\ref{finiteness-thm} we formulate only a part of \cite[Theorem~2 of \S10]{MR893813}. In contrast to \cite{MR893813}, we do not use the notion of successive minima explicitly.}

\begin{theorem}
	\label{3/2-thm}
	Let~$d\ge 2$ and let $K\subseteq\R^d$ be $d$-dimensional, $o$-symmetric, convex, and compact. Let~$\bL$ be a lattice of rank $d$ in $\R^d$. Assume that $K$ contains $d$ linearly independent vectors of $\bL$. Then $(3/2)^{d-2}K$ contains a basis of $\bL$. 
\end{theorem} 

Part~\myeqref{minkowski} of the following theorem is the famous first fundamental theorem of Minkowski; see \cite[Theorem~1 of \S5]{MR893813}. Part~\myeqref{d!} follows directly from \cite[Theorem~5 of \S14]{MR893813}, while part~\myeqref{basis-in-polar} is a straightforward consequence of part~\myeqref{d!} and Theorem~\ref{3/2-thm}.

\begin{theorem}
	\label{o-sym-tools-thm}
	Let~$d \ge 2$ and let $K\subseteq\R^d$ be $d$-dimensional, $o$-symmetric, compact, and convex. Let~$\bL$ be a lattice of rank $d$ in $\R^d$. If $\intr(K) \cap \bL = \{o\}$, the following statements hold: 
	\begin{enumerate}[(a)]
		\item \label{minkowski} $\vol(K) \le 2^d \det(\bL)$. 
		\item \label{d!} $(d!)^2 K^\circ$ contains $d$ linearly independent vectors of $\bL^\ast$. 
		\item \label{basis-in-polar} $(3/2)^{d-2} (d!)^2 K^\circ$ contains a basis of the lattice $\bL^\ast$.
	\end{enumerate}
\end{theorem}

Note that, for $T \subseteq \R^d$ and a lattice $\bL$ of rank $d$, the set $W(T,\bL)$ from~\myeqref{eq:def:WKL} can be also given as 
\begin{align}\label{eq:equiv-char-WKL}
  W(T,\bL) = \bL^\ast \cap \intr (D(T)^\circ )\setminus\{o\}.
\end{align}
This follows  from the well-known equality $w(T,u)=h(D(T),u)$ and the straightforward equivalence $h(D(T),u)<1$ $\Leftrightarrow$ $u\in\intr (D(T)^\circ )$. See  Figure~\ref{fig:width:AveL12:strips} in the introduction for an illustration.
In view of this observation, the following lemma can be used to limit the possible shapes of $W(T,\bL)$. It will be employed both in this section and in the proof of Theorem~\ref{finiteness-thm}.

\begin{lemma} \label{thm:if-T-covers-and-not-flat} 
  Let~$(\bM,\bL,T) \in \cT^d$ and let $T$ be $d$-dimensional.
	Then $(2 \bL^\ast) \cap \intr( D(T)^\circ) = \{o\}$.
\end{lemma}
\begin{proof}
	Assume the contrary, that is, there exists a $u \in \bL^\ast\setminus\{o\}$ such that $2 u \in \intr( D(T)^\circ)$. Appropriately translating $T$ by a vector in $\bM$ we assume that $o \in T$ and $h(T,u)=w(T,u)$. Since $\dim(T)=d$  we have $w(T,u) > 0$. The assumption $2 u \in \intr(D(T)^\circ)$ means  $2 w(T,u) < 1$. We choose~$x \in F(T,u)$. By construction one has $0 \le \sprod{t}{u} \le w(T,u)$ for each $t \in T$ and $\sprod{x}{u}=w(T,u)$. In view of $\bM = \bL+T$ there exist $s \in \bL$ and $t \in T$ such that $2x = s + t$. If $\sprod{s}{u} \le 0$ we obtain
	\(2 w(T,u) = \sprod{2x}{u} = \sprod{s+t}{u} \le w(T,u),
	\)
	a contradiction to $w(T,u) > 0$. Otherwise $\sprod{s}{u} \ge 1$, and this yields
	\(1 > 2 w(T,u) = \sprod{2x}{u} = \sprod{s+t}{u} \ge 1,\)
	a contradiction. 
\end{proof}

Now we have gathered all tools to prove Theorem~\ref{finiteness-thm}.

\begin{proof}[Proof of Theorem~\ref{finiteness-thm}] The assertions are clear for $d=1$, so assume $d\geq 2$.

	\myeqref{item:finiteness-thm:III}: In view of Lemma~\ref{thm:if-T-covers-and-not-flat}, one has $\left( \frac{1}{2} \conv(W(T,\bL)) \right) \cap \bL^\ast = \{o\}$. By Theorem~\ref{o-sym-tools-thm}\,\myeqref{basis-in-polar}, the set $2 (3/2)^{d-2} (d!)^2 \conv(W(T,\bL))^\circ$ contains a basis $b_1,\ldots,b_d$ of $\bL$.  Let $b_1^\ast,\ldots,b_d^\ast$ be the dual basis of~$b_1,\ldots,b_d$. Polarization of the inclusion $2 (3/2)^{d-2} (d!)^2 \conv(W(T,\bL))^\circ \supseteq \{\pm b_1,\ldots,\pm b_d\}$ yields the inclusion $\conv(W(T,\bL)) \subseteq 2 (3/2)^{d-2} (d!)^2 \sum_{i=1}^{d}[-1,1]b_i^\ast$.

        \myeqref{item:finiteness-thm:I}: Since $o \not \in W(T,\bL)$ and $W(T,\bL) \cup \{o\} = \bL^\ast \cap \intr (D(T)^\circ)$, the inequality $\card{W(T,\bL)} < 4^d$ is equivalent to $\card{\bL^\ast \cap \intr (D(T)^\circ)  } \le 4^d$. We show this by contradiction. Assume that $\bL^\ast \cap \intr( D(T)^\circ)$ contains more than $4^d$ elements. Then this set contains two distinct elements $z_1$ and $z_2$ which coincide modulo $4 \bL^\ast$.  Let~$u : = \frac{1}{4} (z_1-z_2) \in \bL^\ast \setminus \{o\}$. In view of the central symmetry and convexity of $\intr (D(T)^\circ)$, we have $2 u = \frac{1}{2} z_1+ \frac{1}{2} (-z_2)\in W(T,\bL)$, which contradicts Lemma~\ref{thm:if-T-covers-and-not-flat}.

	\myeqref{item:finiteness-thm:II}: The polytope $\conv(W(T,\bL))$ is a proper subset of the polytope~$D(T)^\circ$, which gives the inequality $\vol(\conv(W(T,\bL))) < \vol(D(T)^\circ)$. Lemma~\ref{thm:if-T-covers-and-not-flat} and  Minkowski's first theorem (Theorem~\ref{o-sym-tools-thm}\,\myeqref{minkowski}) applied to the $d$-dimensional, $o$-symmetric polytope~$D(T)^\circ$ and the lattice~$2 \bL^\ast$ yield the inequality~$\vol(D(T)^\circ) \le 4^d \det(\bL^\ast)$. Finally, recall that $\det(\bL^\ast)=1 /\det(\bL)$, giving the equality~$4^d \det(\bL^\ast)=4^d /\det(\bL)$.
	\end{proof}
	
\begin{remark}
	The proof of the upper bound on $\card{W(T,\bL)}$ above is based on a version of the so-called parity argument; see, for example, \cite[p.~1613]{MR3106473} for a recent usage.
\end{remark}

\begin{remark}[Boundedness of $T$ relative to $\bL$]
  Theorem~\ref{finiteness-thm} gives a result on the finiteness of the possible shapes of $W(T,\bL)$, but not of $T$. In fact, we cannot have a result on finitely many shapes of $T$. This is seen from Example~\ref{AveL12:example}, in which~$\card{T}$ can be arbitrarily large. Nevertheless,  we have a `boundedness' assertion on $T$ as follows: In the situation of Theorem~\ref{finiteness-thm}, there are $d$ linearly independent vectors in $W(T,\bL)$ and thus in $D(T)^\circ$. Applying Theorem~\ref{3/2-thm} to $D(T)^\circ$, we conclude that $(3/2)^{d-2} D(T)^\circ$ contains a basis of the lattice~$\bL^\ast$. After appropriately changing coordinates in $\R^d$, this basis is $e_1,\ldots,e_d$ and so $\bL^\ast = \bL = \Z^d$. This yields $w(T,e_i) \le (3/2)^{d-2}$ for every $i \in \{1,\ldots,d\}$. Thus, after the mentioned change of coordinates, $\bL=\Z^d$ and $T$ is contained in a translation of the box $[0,(3/2)^{d-2}]^d$, whose size depends only on $d$.
\end{remark}

\goodbreak
\section{Proof of Theorem~\ref{thm:complete-classification}}\label{sec:proof:classification}

The proof of Theorem~\ref{thm:complete-classification}, in particular the proof of implication \myeqref{thm-item:complete-classification:ST-give-NontrHomPair}\,$\Rightarrow$\,\myeqref{thm-item:complete-classification:T-has-width-1}, will need some preparations. We sketch the two main steps before we proceed: Our first goal is Lemma~\ref{key-to-enumeration} which implies that if $(\bM,\bL,T)\in\cT^2$ allows for nontrivially homometric pairs of lattice-convex sets, then $T$ can be `cut out' by a Dirichlet cell and its lattice width with respect to $\bM$ is necessarily $1$, $2$, or $3$. The second key result is Lemma~\ref{lem-enum-T-given-bL}, which shows that, up to translations, such $T$ is contained in a finite list of sets that can be explored via a computer search.

We start our preparations with the following lemma, in which we establish some basic conditions on $S$, $T$ and $W(T,\bL)$, the condition~\myeqref{UT:three:dir} on $W(T,\bL)$ being the most important one. 

\begin{lemma} \label{lem:three-directions:in:sU_T}
	Let~$d=2$ and $(\bM,\bL,T) \in \cT^d$. Let~$S \subseteq \bL$ be finite and such that $S \oplus T$ and $S \oplus (-T)$ form a pair of nontrivially homometric $\bM$-convex sets. Then the following conditions hold:
	\begin{enumerate}[(a)]
		\item\label{S:bL-conv:two-dim} $S$ is $\bL$-convex and two-dimensional.
		\item\label{T:two-dim}  $T$ is two-dimensional.
		\item\label{UT:three:dir} $W(T,\bL)$ contains three pairwise nonparallel vectors. 
	\end{enumerate}
\end{lemma}
\begin{proof}
	The $\bL$-convexity of $S$ follows from Lemma~\ref{lem:S+T:convex=>S:convex}. We have $\dim(S)=2$ for otherwise $S$ would remain unchanged under a point reflection which exchanges the endpoints of the (possibly degenerated) segment $\conv(S)$, so $S$ would be centrally symmetric. Hence by Proposition~\ref{prp:dir:sum:hom}\,\myeqref{S+T,S+(-T):nontr:hom:metr}, the pair $S \oplus T$, $S \oplus (-T)$ would be trivially homometric, a contradiction. So  $S$ is two-dimensional. The same arguments show that $T$ is two-dimensional, so we have established~\myeqref{S:bL-conv:two-dim} and~\myeqref{T:two-dim}. 

	Now observe that by Theorem~\ref{thm:ABC}, every edge of $\conv(S)$ has a normal vector belonging to~$W(T,\bL)$. We show~\myeqref{UT:three:dir} by contradiction, so assume that $W(T,\bL)$ does not contain three pairwise nonparallel vectors. Then $\conv(S)$ is a parallelogram and, by this, centrally symmetric. Since $S$ is $\bL$-convex, the latter implies that $S$ is centrally symmetric. By Proposition~\ref{prp:dir:sum:hom}\,\myeqref{S+T,S+(-T):nontr:hom:metr}, the sets $S \oplus T$ and $S \oplus (-T)$ are trivially homometric, which is a contradiction.
\end{proof}

Having established condition \myeqref{T:two-dim} on $T$ and condition \myeqref{UT:three:dir} on $W(T,\bL)$ in Lemma~\ref{lem:three-directions:in:sU_T}, the structure of a planar tiling $(\bM,\bL,T)$ generating nontrivially homometric pairs $S \oplus T$ and $S \oplus (-T)$ can be specified even more precisely. This is done in Lemma~\ref{key-to-enumeration} below, but first we need more auxiliary results, some of them relying on statements from the geometry of numbers specific to dimension two.

\begin{proposition} 
  \label{prp:o-sym:contains-a-bit-more-than-a-basis} 
  Let~$\bL$ be a lattice of rank two in $\R^2$. Let~$K$ be a two-dimensional, $o$-symmetric, and $\bL$-convex set such that there exists no basis $b_1,b_2$ of the lattice $\bL$ satisfying $[o,b_1] + [o,b_2] \subseteq \conv(K)$. Then
   $\conv(K) = \conv \left(\{ \pm k b_1, \pm b_2\}\right)$ for some basis $b_1,b_2$ of $\bL$ and some $k \in \N$.
\end{proposition}
\begin{proof} After possibly changing coordinates we have $\bL=\Z^2$. By Theorem~\ref{3/2-thm} and since $d=2$, we can choose a basis $b_1, b_2$ of $\integer^2$ in $K$. Possibly applying some unimodular transformation to~$\bL$ (and thus to $K$) we assume $b_1 = e_1$ and  $b_2 = e_2$. We show $K \subseteq \Z\times\{0\}\cup \{0\} \times  \Z$ by contradiction. Assume that there exists a point $p \in K$ not belonging to $\Z \times \{0\} \cup \{0\} \times  \Z$. After possibly changing coordinates using only reflections with respect to coordinate axes, we have $p \in \N^2$. Then $[0,1]^2 \subseteq \conv(\{o,e_1,e_2,p\}) \subseteq \conv(K)$, which contradicts the assumptions. We thus have $K \subseteq  \Z \times \{0\} \cup \{0\} \times \Z$. Clearly, $2 e_1$ and $2 e_2$ cannot be both contained in $K$, for otherwise their convex combination $e_1 + e_2$ belongs to $K$ and we get $[0,1]^2 \subseteq \conv(K)$, which contradicts the assumption. Possibly interchanging the roles of $e_1$ and $e_2$, we have $2 e_2 \not\in K$. Then~$K \subseteq \Z \times \{0\} \cup \{0\} \times \{-1,0,1\}$ and the assertion is established.
\end{proof}

\newcommand{\flt}{\operatorname{Flt}}

A $d$-dimensional, compact, and convex subset of $\R^d$ is called a \emph{convex body} in $\R^d$. For a lattice $\bL$ of rank $d$, one can consider the so-called \emph{flatness constant}:
\[
	\flt(d):= \sup \setcond{w(K,\bL)}{K \ \text{is a convex body in} \ \R^d \ \text{and} \ \intr(K) \cap \bL = \emptyset}.
\]
Clearly, $\flt(d)$ is independent on the choice of $\bL$. It is known that $\flt(d)$ is finite for every $d \in \N$. Results providing upper bounds on $\flt(d)$ are called \emph{flatness theorems}. In dimension $d=2$ the flatness constant is known exactly.\footnote{Currently, $d=2$ is the only dimension, in which the flatness constant is known exactly.}

\begin{theorem}[Exact flatness theorem in dimension two; \cite{0708.52002}] \label{flatness-thm}
	$\flt(2) = 1 + 2 / \sqrt{3}$.
\end{theorem}

\begin{theorem}[\cite{arxiv:1003.4365}] \label{thm:area-vs-width-ineq}
  Let~$\bM$ be a lattice of rank two in $\R^2$ and let $K$ be a two-dimensional compact convex set $K$ satisfying $\intr(K)
  \cap \bM = \emptyset$ and $1< w(K,\bM) \le 2$. Then 
  \begin{align*}
    \label{area-vs-width-ineq}
    \frac{\vol(K)}{\det(\bM)} \le \frac{w(K,\bM)^2}{2(w(K,\bM)-1)}.
  \end{align*}
\end{theorem}

The following lemma is the first key result that makes the computer enumeration possible. 

\begin{lemma} \label{key-to-enumeration}
	Let~$(\bM,\bL,T) \in \cT^2$, let $T$ be two-dimensional and let $W(T,\bL)$ contain three pairwise nonparallel vectors. Then for some basis $b_1^\ast, b_2^\ast$ of $\bL^\ast$, for the dual basis $b_1,b_2$ of $b_1^\ast,b_2^\ast$, and for the triangle $\Delta:= \conv(\{o,b_1,b_2\})$ the following statements hold:
  \begin{enumerate}[(a)]
    \item \label{thin-in-three-directions} $w(T,b_1^\ast)< 1$, $w(T,b_2^\ast) < 1$, and $w(T,b_1^\ast + b_2^\ast) < 1$.
    \item \label{T-is-cut-out-form-cell} There exists $v \in \R^2$ such that
    \begin{align*} 
      T = \bM \cap \bigl(v+ (0,1] b_1 + (0,1] b_2 \bigr).
    \end{align*}
    \item \label{width-of-Delta-and-T} $w(T,\bM) \le w(\Delta,\bM) -1$. 
    \item \label{width-of-Delta-and-finiteness} $w(\Delta,\bM) \in \{2,3,4\}$ and
    \begin{equation}
      \label{det-ratio-range}
      \frac{\det(\bL)}{\det(\bM)} \in
      \begin{cases}
        \{7,\ldots,18\} & \text{if} \ w(\Delta,\bM) =3, 
        \\
        \{12,\ldots,16\} & \text{if} \ w(\Delta,\bM) =4.
      \end{cases}
    \end{equation}
  \end{enumerate}
\end{lemma}
\begin{proof}
  \myeqref{thin-in-three-directions}: There exists a basis $b_1^\ast,b_2^\ast$ of $\bL^\ast$ lying in $\intr (D(T)^\circ)$ such that $b_1^\ast+b_2^\ast$ is also in $\intr (D(T)^\circ)$. Indeed, if the latter was not the case, Proposition~\ref{prp:o-sym:contains-a-bit-more-than-a-basis} applied to the lattice $\bL^{\ast}$ and $K=\bL^\ast \cap \conv(\intr(D(T)^\circ))$ would yield that $W(T,\bL)$ does not contain three pairwise nonparallel vectors, which contradicts the assumption. Thus, $b_1^\ast, b_2^\ast, b_1^\ast + b_2^\ast \in W(T,\bL)$, which gives \myeqref{thin-in-three-directions}.

  \myeqref{T-is-cut-out-form-cell}: After possibly changing coordinates we have $b_1^\ast = e_1$ and $b_2^\ast = e_2$; thus $\bL^\ast = \Z^2$, $\bL=\Z^2$, $b_1= e_1$,  and $b_2=e_2$. The conditions $b_1^\ast, b_2^\ast,b_1^\ast+b_2^\ast\in\intr(D(T)^\circ)$ translates to $w(T,e_1)<1$, $w(T,e_2)<1$, and $w(T,e_1 + e_2)<1$. For $h_i := \max_{t\in T}\sprod{T}{e_i}$ with $i \in \{1,2\}$ and $h:= \max_{t\in T}\sprod{t}{e_i+e_2}$, this gives
  \begin{align}\label{eq:width123:2}
    T \subseteq \setcond{(x_1,x_2) \in (h_1-1,h_1] \times (h_2-1,h_2]}{ h - 1 <
    x_1 + x_2 \le h}.
  \end{align}
  In particular, we have $T \subseteq \bM\,\cap\,(h_1-1,h_1] \times (h_2-1,h_2]$, so $T\subseteq \bM\cap (v + (0,1]^2)=\bM\cap (v + (0,1]b_1+(0,1]b_{2})$ for $v:=(h_1-1,h_2-1)$. We show~\myeqref{T-is-cut-out-form-cell} by contradiction.  Assume that $x\in\bM$ belongs to $v+(0,1]^2$, but not to $T$. Each translation $z+T$ of $T$ with $z \in \Z^2 \setminus \{o\}$ is a subset of $z+ v + (0,1]^2$. The sets $v+(0,1]^2$ and $z+v+(0,1]^2$ are disjoint because $\R^2$ is the disjoint union of all integral translations of $(0,1]^2$. Consequently, $x\notin z+v+(0,1]^2$ and, by this, $x\notin z+T$. This shows that $x\notin\Z^2 + T = \bL + T=\bM$, which contradicts the assumption $x\in\bM$.

\myeqref{width-of-Delta-and-T}: We claim that $1> h_1 + h_2 - h\geq 0$. Indeed, the inequality $h \le h_1 + h_2$ is valid since for~$i \in \{1,2\}$ the value $h_i$ is the maximum of all $\sprod{t}{e_i}$ with $t \in T$, while $h$ is the maximum of all sums $\sprod{t}{e_1} + \sprod{t}{e_2}$ with $t \in  T$. For showing $h_1 + h_2 - 1<h$, we choose  $(x_1,h_2), (h_1,x_2) \in T$ with suitable $x_{1},x_{2}\in\R^2$. We have $x_i > h_i - 1$ for each $i \in \{1,2\}$ due to~\myeqref{eq:width123:2}. Together with the inequalities $x_1 + h_2 \le h$ and $h_1 + x_2 \le h$ we get $h \ge \frac{1}{2} ( x_1 + x_2 + h_1 + h_2) > h_1 + h_2 -1$.

A~translation of the left and the right hand side of \myeqref{eq:width123:2} yields the inclusion
  \begin{align*}
    T - v \subseteq \setcond{(x_1,x_2) \in (0,1]^2}{ 1 + h - h_1 - h_2
    < x_1 + x_2 \le 2 + h - h_1 - h_2}.
  \end{align*}
  The topological closure of the right hand side of the latter inclusion is the polygon
  \[
    H_\alpha : = \setcond{ (x_1,x_2) \in [0,1]^2}{ \alpha \le x_1 + x_2 \le \alpha + 1},
  \]
  where $\alpha := 1 + h - h_1 - h_2$ satisfies $0 < \alpha \le 1$. The polygon $H_\alpha$ is a hexagon for $0 < \alpha < 1$ and a triangle otherwise. With $\Delta = \conv(\{o,b_1,b_2\}) = \conv(\{o,e_1,e_2\})$ it is straightforward to check that $H_\alpha = (\alpha,\alpha) - \alpha \Delta + (1-\alpha) \Delta$. This implies $D(H_\alpha) = D(\Delta)$. Thus, $D(H_{\alpha})$ is independent of $\alpha$ and $w(H_\alpha,u) =\max_{x\in D(H_{\alpha})}\sprod{x}{u}=\max_{x\in D(\Delta)}\sprod{x}{u}= w(\Delta,u)$ for each $u \in \R^2$. Since an appropriate translation of $T$ is contained in $\intr (H_\alpha)$, we have $D(T) \subseteq D(\intr (H_\alpha)) = \intr (D(H_\alpha))$ and $w(\Delta,u) = w(H_\alpha, u) > w(T,u)$ for every $u \in \R^2 \setminus \{o\}$. Moreover, the vertices of $D(H_\alpha) = D(\Delta)$ and $D(T)$ lie in $\bM$. Hence the strict inequality $w(T,u) < w(\Delta, u)$ can be improved to $w(T,u) \le w(\Delta, u) - 1$ for each $u \in \bM^\ast \setminus \{o\}$, giving \myeqref{width-of-Delta-and-T}.

\myeqref{width-of-Delta-and-finiteness}: The points of $[0,1]^2$ not covered by $H_\alpha$ lie in the triangles $\alpha \Delta$ and $(1,1) - (1-\alpha) \Delta$, where the second triangle is degenerated to a point if $\alpha=1$. We show \myeqref{width-of-Delta-and-finiteness} by distinguishing two cases according to which of the two triangles $\alpha \Delta$ and $(1-\alpha) \Delta$ is larger. 

\emph{Case 1: $\alpha \ge \frac{1}{2}$.} The set $T$ is contained in the polytope $H_\alpha +v$, and so by \myeqref{T-is-cut-out-form-cell} the interior of $\alpha \Delta + v$ does not contain points of $\bM$. Consequently, the interior of the subset $K:=\frac{1}{2} \Delta$ of $\alpha \Delta + v$  does not contain points of $\bM$ either. Thus,  Theorem~\ref{flatness-thm} yields
\(  w(\Delta,\bM) = 2 w(K,\bM) \le \textstyle 2 \left(1+
  {2}/{\sqrt{3}} \right)
\).
Since the vertices of $\Delta$ belong to $\Z^2 = \bL \subseteq \bM$, we have $w(\Delta,\bM) \in \N$, which implies $w(\Delta,\bM) \le \floor{2 (1+ {2}/{\sqrt{3}})} = 4$. Furthermore, in view of~\myeqref{width-of-Delta-and-T} and $w(T,\bM) \in \N$, we also have $w(\Delta,\bM) \ge 2$. In fact, $w(\Delta,\bM)=1$ would imply $w(T,\bM)=0$, which contradicts the full-dimensionality of $T$. Thus, $w(\Delta,\bM) \in \{2,3,4\}$. 

Next we show the upper bounds on $\det(\bL)/ \det(\bM)$ contained in \myeqref{det-ratio-range}. We have $\det(\bL) = 1$ and $\vol(K) = 1/8$. Thus, $\det(\bL)/ \det(\bM) = 8 \vol(K) / \det(\bM)$. Assume $w(\Delta,\bM) \in \{3,4\}$. For bounding $\vol(K)/ \det(\bM)$, we can use Theorem~\ref{thm:area-vs-width-ineq} for $K$: Since $w(K,\bM)=\frac{1}{2}w(\Delta,\bM) \in \{3/2,2\}$, the set $K$ fulfills the assumptions of this theorem. We obtain
  \begin{align*}
    8 \frac{\vol(K)}{\det(\bM)} \le 8 \frac{w(K,\bM)^2}{2 (w(K,\bM) - 1)} 
= \frac{2 w(\Delta,\bM)^2}{w(\Delta,\bM) - 2} 
    = 
    \begin{cases}
      18 & \text{if} \ w(\Delta,\bM) = 3,
      \\ 16 & \text{if} \ w(\Delta,\bM) =4.
    \end{cases}
  \end{align*}
  This yields the desired upper bounds on $\det(\bL) / \det(\bM)$. 
  
  To conclude Case~1, it remains to show the lower bounds on $\det(\bL)/\det(\bM)$ from \myeqref{det-ratio-range}. We have $\det(\bL)/ \det(\bM) = 1/\det(\bM)= \det(\bM^\ast)$. For finding a lower bound on $\det(\bM^\ast)$ we use Minkowski's first theorem (Theorem~\ref{o-sym-tools-thm}\,\myeqref{minkowski}) for the lattice $\bM^\ast$. A direct computation shows $D(K)^\circ = 2 D(\Delta)^\circ = 2 \conv(\{\pm e_1,\pm e_2, \pm (e_1 + e_2) \})$ and, consequently, $\vol(D(K)^\circ) = 12$. Using standard facts about the width and the polarity, it follows that the interior of $w(K,\bM) D(K)^\circ$ consists of vectors $u \in \R^d$ satisfying $w(K,u) < w(K,\bM)$. Thus, by definition of $w(K,\bM)$, the interior of $w(K,\bM) D(K)^\circ$ does not contain nonzero vectors of $\bM^\ast$. So Minkowski's first theorem can be applied to $w(K,\bM) D(K)^\circ$; taking into account $\vol(D(K)^\circ) = 12$ we obtain $4 \det(\bM^\ast) \ge \vol(w(K,\bM) D(K)^\circ) = 12 w(K,\bM)^2$. In view of $\det(\bL)/\det(\bM) = \det(\bM^\ast)$, this gives
  \begin{align*}
    \frac{\det(\bL)}{\det(\bM)} \ge 3 w(K,\bM)^2 = \frac{3}{4} w(\Delta,\bM)^2  = \begin{cases}
      \frac{27}{4} & \text{if} \ w(\Delta,\bM) = 3,
      \\ 12 & \text{if} \ w(\Delta,\bM) =4.
    \end{cases}
  \end{align*}
  The lattice $\bL$ is a sublattice of $\bM$. It is well known that in this case $\det(\bL) / \det(\bM)$ is a natural number. Thus, the lower bound $27/4$ in the case $w(\Delta,\bM)=3$ can be rounded up to $7$. This yields the lower bounds on $\det(\bL) / \det(\bM)$ contained in \myeqref{det-ratio-range}.

  \emph{Case 2: $\alpha \le \frac{1}{2}$.} In this case completely analogous arguments can be applied to the triangle $(1,1) - (1-\alpha) \Delta$ instead of the triangle $\alpha \Delta$ to get \myeqref{width-of-Delta-and-finiteness}. 
\end{proof}

In view of Lemma~\ref{lem:three-directions:in:sU_T} and Lemma~\ref{key-to-enumeration}, we can prove Theorem~\ref{thm:complete-classification}\,\myeqref{thm-item:complete-classification:ST-give-NontrHomPair}\,$\Rightarrow$\,\myeqref{thm-item:complete-classification:T-has-width-1} by distinguishing the three cases $w(\Delta,\bM)=2$, $w(\Delta,\bM) = 3$ and $w(\Delta,\bM) =4$. We will see that in the case $w(\Delta,\bM)=2$ (which means $w(T,\bM)=1$), the assertion will follow from results in \cite{AveL12}. When $w(\Delta,\bM)  \in \{3,4\}$, we use the bounds on $\det(\bL)/\det(\bM)$ from Lemma~\ref{key-to-enumeration}\,\myeqref{width-of-Delta-and-finiteness}. These bounds and the following statement enable us to fix $\bM$ and to carry out an  computer-assisted enumeration of all possible lattices $\bL$ using Magma.

\begin{proposition} \label{finite:choice:bL}
	Let~$\bM$ be a lattice of rank $d$ in $\R^d$ and let $L \in \N$. Then there exist only finitely many rank $d$ sublattices $\bL$ of $\bM$ with $\det(\bL)/\det(\bM) = L$. 
\end{proposition}
This proposition is folklore and can be derived using transformation of $d$$\times$$d$ integral  matrices into Hermite normal form; see \cite[Theorem~2.2 and the preceding paragraph%
]{MR1373675}. We rely on some arguments of the proof later on in our computer enumeration, therefore we give details.
	
\begin{proof}[Proof of Proposition~\ref{finite:choice:bL}]
	        After possibly changing coordinates we have $\bM=\Z^d$. Then each $\bL$ as in the assertion can be given as $\bL = B (\Z^d)$ with a suitable $B \in \Z^{d \times d}$ having determinant $L$. For every $d$$\times$$d$ unimodular matrix $U \in \Z^{d \times d}$ one has $U(\Z^d) = \Z^d$. Hence $\bL = B U (\Z^d)$. We can choose $U$ such that $H = B U$ is the Hermite normal form of $B$; see \cite[Chapter~4]{MR874114}. The condition that the $d$$\times$$d$ matrix $H$ is in Hermite normal form means that $H$ is lower triangular, the elements of $H$ are nonnegative and each row of $H$ has a unique maximum element, which is the element lying on the main diagonal of~$H$. We have $\det(H) = \det(BU) = \det(B) = L$. Clearly, there are only finitely many $d$$\times$$d$ matrices in Hermite normal form with the determinant $L$.
\end{proof}

To deal with the case $w(\Delta,\bM) \in \{3,4\}$, we will go through all possible $\bL$ and, for each choice of $\bL$, we will enumerate, up to translations, all sets $T$ with $(\bM,\bL,T) \in \cT^2$ and $T$ given as in Lemma~\ref{key-to-enumeration}\,\myeqref{T-is-cut-out-form-cell}. This enumeration will rely on the following second key lemma. The latter is formulated for arbitrary~$d \ge 2$, though for proving Theorem~\ref{thm:complete-classification} we need the case $d=2$ only. 

\begin{lemma}
  \label{lem-enum-T-given-bL}
  Let~$(\bM,\bL,T)\in\cT^{d}$ with $\bM=\Z^d$. Let~$b_1,\ldots,b_d \in \bL$ be a basis of $\bL$ such that the matrix $B \in \Z^{d \times d}$  with columns $b_1,\ldots,b_d$, in that order, satisfies $L:=\det(B)>0$. Let  $A$ be the adjugate matrix of $B$, i.e., $A=L \cdot B^{-1}\in \Z^{d \times d}$.  For each $i \in \{1,\ldots,d\}$, let $n_i$ be the greatest common divisor of the entries in the $i$-th row of $A$. Assume
  \[
    T=\bM \cap \Bigl(v + \sum_{i=1}^d (0,1] b_i \Bigr)
  \]
for some $v \in \R^d$.    
 Then $T$ coincides, up to a translation by a vector in $\bL$, with
  \begin{align}
    \label{T-q-def}
    T_q:=\bM \cap \sum_{i=1}^d [ (q_i + n_i)/L, (q_i+L)/L ] b_i
  \end{align}
  for some $q=(q_1,\ldots,q_d)$ that satisfies 
  \begin{align}
    \label{q-range}
    q_i \in \{0,\ldots, L-1\}\cap n_i \Z \qquad (i\in \{1,\ldots,d\}).
  \end{align}
\end{lemma}
\begin{proof} We first show  the exact equality $T=T_q$ for some $q$ with $q_i \in n_i \Z$ for each $i \in \{1,\ldots,d\}$. In a second step we derive the equality up to translations with the range of $q$ given by \myeqref{q-range}.

Let~$b_1^\ast,\ldots,b_d^\ast$ be the dual basis of $b_1,\ldots,b_d$ and, for $i \in \{1,\ldots,d\}$, let $a_i$ be the $i$-th row of $A$. Thus, $n_i=\gcd(a_i)$ and $a_i= L b_i^\ast$. For each $t \in \bM$  and each $i\in\{1,\ldots,d\}$ the condition $a_i / n_i \in \Z^d$ implies $\sprod{t}{a_i} = n_i \sprod{t}{a_i/n_i} \in n_{i}\Z$. Analogously, $L = \sprod{b_i}{a_i} = n_i \sprod{b_i}{a_i/n_i} \in n_{i}\Z$. 

 Clearly, for any $t\in\bM$ one has $t \in v + \sum_{i=1}^d (0,1] b_i$ if and only if $0 < \sprod{t-v}{b_i^\ast} \le 1$ for every $i \in \{1,\ldots,d\}$.  Using $a_i = L b_i^\ast$, the condition $0 < \sprod{t -v}{b_i^\ast} \le 1$ can be reformulated as  \(\sprod{v}{a_i} <  \sprod{t}{a_i} \le  \sprod{v}{a_i}+L\). Dividing by $n_i$, we get 
\begin{equation} \label{t-a_i-n_i}
  \sprod{v}{a_i} /n_i < \sprod{t}{a_i}/n_i \le  \sprod{v}{a_i}/n_i +  L/n_i,
\end{equation}
where the values $\sprod{t}{a_i}/n_i$ and $L/n_i$ are integer, while $\sprod{v}{a_i} / n_i$ is possibly fractional. So in~\myeqref{t-a_i-n_i} rounding down $\sprod{v}{a_i}/ n_i$ does not change the condition on~$t \in \bM$.  We obtain
\begin{equation} \label{t-a_i-n_i-with-rounding}
  \floor{\sprod{v}{a_i} /n_i} < \sprod{t}{a_i}/n_i \le   \floor{\sprod{v}{a_i}/n_i} + L/n_i.
\end{equation}
Since $\floor{\sprod{v}{a_i} /n_i}$ and $\sprod{t}{a_i}/n_i$ are integers, \myeqref{t-a_i-n_i-with-rounding} can be rewritten as 
\[
  \floor{\sprod{v}{a_i} /n_i} + 1 \le \sprod{t}{a_i}/n_i 
\le
  \floor{\sprod{v}{a_i}/n_i} + L/n_i.
\]
Choosing $q_i = n_i \floor{\sprod{v}{a_i} /n_i} \in n_i \Z$ for each $i$, we get 
\begin{equation}
  \label{T-through-q}
  T = \setcond{t \in \bM}{ q_i + n_i \le \sprod{t}{a_i} \le q_i+L~\text{for each $i\in\{1,\ldots,d\}$}}=T_{q},
\end{equation}
where the last equality uses $T_q$ as in \myeqref{T-q-def} and is straightforward to check in view of  $a_i= L b_i^\ast$.

It remains to show that $T$ coincides with $T_q$, up to a translation by a vector in $\bL$, with some $q$ satisfying \myeqref{q-range}. In view of \myeqref{T-through-q}, one directly computes $T_{q} + b_j = T_{q + L e_j}$ and, similarly, $T_q - b_j = T_{q- L e_i}$ for each $j\in\{1,\ldots,d\}$. In other words, by adding $L$ to $q_j$ we translate $T_q$ by $b_j$ and by subtracting~$L$ from $q_j$ we translate $T_q$ by $-b_j$. Since $L$ is divisible by $n_j$, such a change of $q_j$ does not affect the condition $q_j \in n_j \Z$. Suitably performing the mentioned changes of the components of $q$ finitely many times we obtain $q \in \{0,\ldots,L-1\}^d$, concluding the proof.
\end{proof}

\begin{proof}[Proof of Theorem~\ref{thm:complete-classification}] After possibly changing coordinates we have $\bM=\Z^2$.

\myeqref{thm-item:complete-classification:T-has-width-1}\,$\Rightarrow$\,\myeqref{thm-item:complete-classification:ST-give-NontrHomPair}: Let \myeqref{thm-item:complete-classification:T-has-width-1} hold. After a suitable unimodular transformation we have $a_1=e_1$, $a_2=e_2$, $b_{1}=(k+1,-1)$, $b_{2}=(k,1)$, and $T+a=\{0,\ldots,k\} \times \{0\} \cup \{1,\ldots,k-1\} \times \{1\}$. It was shown in \cite[Theorem 2.5, Corollary 2.6]{AveL12} that \myeqref{thm-item:complete-classification:T-has-width-1}\,$\Rightarrow$\,\myeqref{thm-item:complete-classification:ST-give-NontrHomPair} holds for this choice of $(\bM,\bL,T) \in \cT^2$.

\myeqref{thm-item:complete-classification:ST-give-NontrHomPair}\,$\Rightarrow$\,\myeqref{thm-item:complete-classification:T-has-width-1}: By Lemma~\ref{lem:three-directions:in:sU_T}, both $S$ and $T$ are two-dimensional and noncentrally symmetric and the set $W(T,\bL)$ contains at least three pairwise nonparallel vectors. We borrow the notation from  Lemma~\ref{key-to-enumeration}. This lemma yields $w(\Delta,\bM)\in\{2,3,4\}$.

\emph{Case 1: $w(\Delta,\bM)=2$}. Lemma~\ref{key-to-enumeration}\,\myeqref{width-of-Delta-and-T} and two-dimensionality of $T$ yield $w(T,\bM)=1$, that is, up to a suitable unimodular transformation, a translate of $T$ is given by $\{0,\ldots,k\} \times \{0\} \cup \{1,\ldots,\ell \} \times \{1\}$ with some $\ell\in\Z$ satisfying $k \ge \ell\geq 0$. We have $k > \ell$, since otherwise $T$ is centrally symmetric. Now the implication~\myeqref{thm-item:complete-classification:ST-give-NontrHomPair}\,$\Rightarrow$\,\myeqref{thm-item:complete-classification:T-has-width-1} follows directly from \cite[Theorem~2.5 and Corollary 2.6]{AveL12}.

\emph{Case 2: $w(\Delta,\bM) \in \{3,4\}$}. Due to Lemma~\ref{key-to-enumeration}, it is sufficient consider the finitely many sublattices $\bL$ of $\bM$ with $\det(\bL)\in\{7,\ldots,18\}$ for $w(\Delta,\bM) = 3$ and $\det(\bL)\in\{12,\ldots,16\}$ for $w(\Delta,\bM)=4$. Let~$\bL$ be such a sublattice and let $b_{1},b_{2}\in\bL$ and $b_1^\ast, b_2^\ast\in\bL^{*}$ be as in Lemma~\ref{key-to-enumeration}\,\myeqref{thin-in-three-directions} and~\myeqref{T-is-cut-out-form-cell}. We can change coordinates using Proposition~\ref{change:prp} in such a way that $(\bM,\bL,T)$ is replaced by~$(U(\bM),U(\bL),U(T))$, but $U(\bM)$ is still $\Z^2$, i.\,e., we use a unimodular matrix $U\in \Z^{2\times 2}$. There exists a unimodular matrix $U$ such that the transpose $B^\top U^\top$ of $UB$ is in Hermite normal form, where~$B$ has columns $b_1$ and $b_2$; see the proof of Proposition~\ref{finite:choice:bL}. So, after a suitable change of coordinates we assume without loss of generality that 
  \begin{align}\label{eq:baisis-in-Hermite-normal-form}
    b_{1}&=(\ell,0),  & b_{2}&=(s,h)
  \end{align}
for some integer values $\ell,s,h \ge 0 $ with $s<h$. Such a triple~$\ell, s, h$  determines the lattice $\bL$ and one has $\ell h=\det(\bL)$. By Lemma~\ref{key-to-enumeration}\,\myeqref{T-is-cut-out-form-cell} and Lemma~\ref{lem-enum-T-given-bL}, the set $T$ coincides up to translations with one of the finitely many sets $T_q$ defined in Lemma~\ref{lem-enum-T-given-bL}. 

Using Magma we performed the following computer search (for details, see Section~\ref{sec:magma-code}). We enumerated all triples of integers~$\ell, s, h \ge 0 $ with $s < h$ and with $\ell h = \det(\bL)\in\{7,\ldots,18\}$. For each such triple we checked the validity of \myeqref{det-ratio-range}. Whenever \myeqref{det-ratio-range} was fulfilled, we enumerated all $T_q$ as in Lemma~\ref{lem-enum-T-given-bL} and checked the validity of $\dim(T_q)=2$ and the condition $w(T_q,b_1^\ast + b_2^\ast)< 1$ from Lemma~\ref{key-to-enumeration}\,\myeqref{thin-in-three-directions}. 

The search showed that the sets $T_q$ which pass all mentioned tests coincide, up to affine unimodular transformations, with  $\{o,\pm e_1,\pm e_2,\pm(e_1+e_2)\}$. The latter set is centrally symmetric. Since $T$ must be noncentrally symmetric, such sets $T_q$ can be discarded, concluding the proof.
\end{proof}

\begin{remark}\label{rem:centrally-symmetric-T}
	The above arguments can be used to provide an explicit description of all tilings $(\bM,\bL,T) \in \cT^2$, where $T$ is two-dimensional and $W(T,\bL)$ contains at least three pairwise nonparallel vectors (without restricting $T$ to be centrally symmetric). Up to unimodular transformations of~$\bM$ (and thus $\bL$ and $T$) and translations of~$T$, apart from the tiling given in Example~\ref{AveL12:example}, the only remaining case to be analyzed is $w(T,\bM)=1$ and $T =  \{0,\ldots,k\} \times \{0,1\}$ (with $k \in \N$). In this case, for an appropriate choice of $\bL$, the set $W(T,\bL)$ contains three pairwise nonparallel vectors; see Figures~\ref{fig:centrally-symmetric-T} (b), (c), and (d) for an illustration for $k=2$. Furthermore, the computer enumeration that we performed yields one more example (again, up to affine transformations that preserve $\bM$) that is depicted in  Figures~\ref{fig:centrally-symmetric-T} (f), (g), and (h). 
\end{remark}

\begin{figure}[htb]
  \begin{tabular}{p{3.5cm}p{3.5cm}p{3.5cm}p{3.5cm}}
    (a)\par\smallskip\qquad\includegraphics{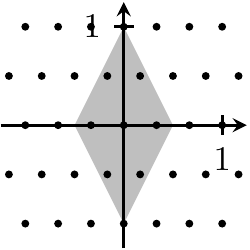}&
    (b)\par\smallskip\includegraphics{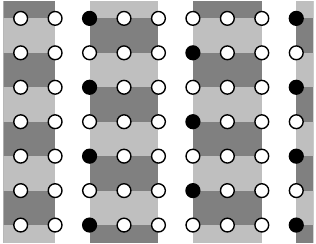}&
    (c)\par\smallskip\includegraphics{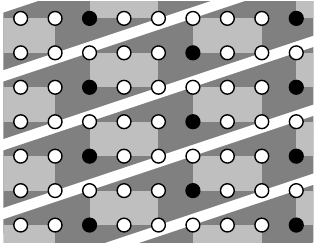}&
    (d)\par\smallskip\includegraphics{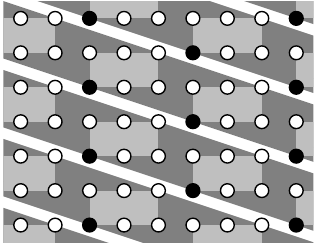}
    \\[2ex]
    (e)\par\smallskip\qquad\includegraphics{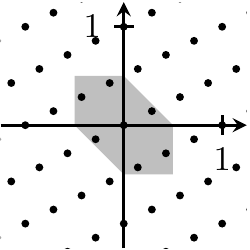}&
    (f)\par\smallskip\includegraphics{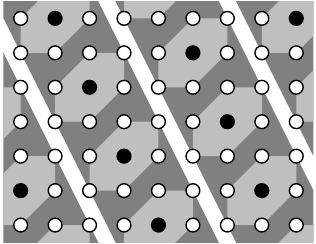}&
    (g)\par\smallskip\includegraphics{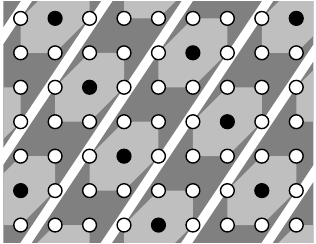}&
    (h)\par\smallskip\includegraphics{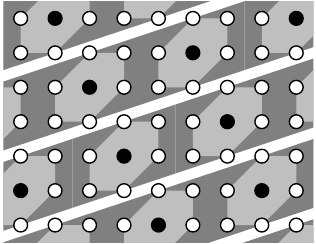}
   \end{tabular}
   \caption{Two examples of tilings $(\bM,\bL,T) \in \cT^2$ with a two-dimensional $T$ and $W(T,\bL)$ consisting of six vectors. The layout of the pictures is the same as in Figure~\ref{fig:width:AveL12:strips} on page \pageref{fig:width:AveL12:strips}.}
   \label{fig:centrally-symmetric-T}
\end{figure}

\goodbreak
\section{Constructions and examples}\label{sec:examples}

\paragraph{A generalization of Example~\ref{AveL12:example} to arbitrary dimension $d\geq 2$.} We construct a generalization of the tilings in Example~\ref{AveL12:example}. The following lemma provides the necessary tools.

\begin{lemma} \label{dual:special:lattice}
	Let~$a = (a_1,\ldots,a_d) \in \Z^d$ and let $r \in \N$. Then the following statements hold:
	\begin{enumerate}[(a)]
		\item\label{lem-item:dual:special:a} $\bL := \setcond{z \in \Z^d}{\sprod{z}{a} \in r \Z}$ is a lattice of rank $d$.
		\item\label{lem-item:dual:special:b} $\bL^\ast = \Z^d + \Z \frac{a}{r}$. 
		\item\label{lem-item:dual:special:c} If $a_1=1$ and $b^*_{i}:=\frac{1}{r}a-\sum_{\ell=i +1}^{d}e_\ell$ for $i \in \{1,\ldots,d\}$, then $b^*_{1},\ldots,b^*_{d}$ is a basis of $\bL^\ast$.
		\item \label{lem-item:dual:special:d} If $a_1=1$, then the vectors $b_1,\ldots,b_d$ given by
		\begin{equation*}
			b_i := 
			\begin{cases}
				a_2 e_1 - e_2 & \text{if} \ i=1,
				\\ (a_{i+1}-a_i) e_1 + e_i - e_{i+1} & \text{if} \ 2 \le i \le d-1,
				\\ (r-a_d) e_1 + e_d & \text{if} \ i=d
			\end{cases}
		\end{equation*}
		form a basis of $\bL$ which is dual to the basis $b_1^\ast,\ldots,b_d^\ast$ in \myeqref{lem-item:dual:special:c}.
	\end{enumerate}
\end{lemma}
\begin{proof}
	\myeqref{lem-item:dual:special:a}: There are several equivalent ways to define lattices. Following~\cite[pp.~279-280]{MR1940576}, the set~$\bL$ is a lattice if and only if $\bL$ is an additive subgroup of $\R^d$ such that some neighborhood of~$o$ contains no points of $\bL \setminus \{o\}$. The latter is clearly fulfilled, so $\bL$ is a lattice. To see that $\bL$ has rank~$d$, observe that  the $d$ linearly independent vectors $r e_1,\ldots,r e_d$ belong to $\bL$.
	
	\myeqref{lem-item:dual:special:b}: Dualization of the left and the right hand side yields that $\bL^\ast = \Z^d + \Z \frac{a}{r}$ is equivalent to~$\bL = (\Z^d + \Z \frac{a}{r})^\ast$. The latter equality is shown as follows:
        \begin{align*}
         x\in \textstyle(\Z^d + \Z \frac{a}{r})^\ast
         \quad\Leftrightarrow\quad & \forall\, z \in \Z^d : \forall\, \ell \in \Z:  \textstyle\sprod{x}{z + \ell \frac{a}{r}}  \in \Z &&\text{}\\
\Leftrightarrow\quad&(\forall\, z \in \Z^d: \sprod{x}{z} \in \Z^d)
\quad\text{and}\quad(\forall\, \ell \in \Z: \textstyle \ell \sprod{x}{\frac{a}{r}} \in \Z)
\\
\Leftrightarrow\quad&x \in \Z^d
\quad\text{and}\quad(\forall\, \ell \in \Z: \textstyle \ell \sprod{x}{\frac{a}{r}} \in \Z)
\\
\quad\Leftrightarrow\quad&x\in\bL,
        \end{align*}
where in the second equivalence one considers the special cases $\ell=0$ and $z =o$.

\myeqref{lem-item:dual:special:c}: Clearly, $\bL^\ast$ has rank $d$ and $b_1^\ast, \ldots, b_d^\ast\in\bL^\ast$ due to~\myeqref{lem-item:dual:special:b}. Thus, it suffices to check that~$\bL^\ast \subseteq \Z b_1^\ast + \cdots + \Z b_d^\ast$. Since $\bL^\ast = \Z e_1 + \cdots + \Z e_d  + \Z \frac{a}{r}$, it is sufficient to show that the vectors $e_1,\ldots,e_d, \frac{a}{r}$ belong to $\Z b_1^\ast + \cdots + \Z b_d^\ast$. We have $\frac{a}{r} = b_d^\ast$ and it remains to consider the vectors $e_1,\ldots,e_d$. Clearly, $e_i = b_{i}^{\ast}-b_{i-1}^\ast$ for each $i \in \{2,\ldots,d\}$. Further, using $a_1=1$ we get $e_1 = a - (a_2 e_2 + \cdots + a_d e_d) \subseteq \Z b_1^\ast + \cdots + \Z b_d^\ast$. It follows that $b_1^\ast,\ldots, b_d^\ast$ is a basis of $\bL^\ast$. 

\myeqref{lem-item:dual:special:d}: The equalities $\sprod{b_i}{b_j^\ast} = \delta_{ij}$ for all $i,j \in \{1,\ldots,d\}$ can be checked straightforwardly.
\end{proof}

\begin{example} \label{AveL12:gen:example}
	Let 
	\begin{align} \label{AveL12:gen:example:prelim}
		d & \ge 2, & k & \in \N, & \bM  &:= \Z^d.
	\end{align}
	We construct $T$ to be a union of $d$ parallel lattice segments, where one of the lattice segments is~$\{0,\ldots, k\} e_1$, which consists of $k+1$ points. The remaining $d-1$ lattice segments are $\{0,\ldots,k-1\} e_1 + e_i$ with $i \in \{2,\ldots,d\}$, each consisting of $k$ lattice points. That is, $T$ has
	\begin{equation} \label{AveL12:gen:example:r}
		r:=dk + 1
	\end{equation}
	elements and is given by 
	\begin{equation} \label{AveL12:gen:example:T}
		T = \{0,\ldots,k\} e_1 \cup \bigcup_{i=2}^d ( \{0,\ldots,k-1\} e_1 + e_i).
	\end{equation}
	Clearly, $T$ is $\bM$-convex, finite, and $d$-dimensional. 
	
	Having fixed $\bM$ and $T$, we want to choose $\bL$ such that $\bM = \bL \oplus T$ holds and with the lattice~$\bL$ defined as in Lemma~\ref{dual:special:lattice}. To this end, we first introduce a linear function from $\R^d$ to $\R$, which sends $\bM$ to $\Z$ and maps the $d$ parallel lattice segments, of which $T$ is comprised, to $d$ consecutive lattice segments in $\Z$. This linear function is defined by prescribing the images of the standard basis $e_1,\ldots,e_d$. We map $e_1$ to $1$, ensuring that the $d$ lattice segments in the definition of $T$, which are all parallel to the vector $e_1$, are sent to lattice segments of $\Z$. Now $\{0,\ldots,k\} e_1$ is mapped to $\{0,\ldots,k\}$. We want the next lattice segment $\{0,\ldots,k-1\} e_1 + e_2$ to be mapped to the lattice segment which follows~$\{0,\ldots,k\}$, so we send $e_2$ to $k+1$. Proceeding iteratively in a similar fashion we see that whenever $e_i$ is sent to~$(i-1) k + 1$ for each $i \in \{1,\ldots,d\}$, the $d-1$ lattice segments~$\{0,\ldots,k-1\} e_1 + e_i$ with $i \in \{2,\ldots,d\}$ are mapped to~$d-1$ consecutive lattice segments of $\Z$. In other words, we have constructed a linear function sends $z \in \Z^d$ to~$\sprod{z}{a}$ with 
	\begin{equation} \label{AveL12:gen:example:a}
		a := \sum_{i=1}^d ((i-1) k + 1) e_i.
	\end{equation}
 This linear function is used to define
	\begin{equation}
		\label{AveL12:gen:example:bL}
		\bL := \setcond{z \in \Z^d}{\sprod{z}{a} \in r \Z}.
	\end{equation}
\end{example} 

	Next we show that the above example satisfies $(\bM,\bL,T) \in \cT^d$ (see Lemma~\ref{gen:example:is:indeed:tiling}) and that  the assumptions of Corollary~\ref{nontrivial-S-cor} are fulfilled for the tiling $(\bM,\bL,T)$ (see Lemma~\ref{lem:sU_T:for:AveL12:gen:example}). Thus, one can find $S$ such that $S \oplus T$ and $S \oplus (-T)$ are nontrivially homometric and $\bM$-convex. 

\begin{lemma}
	\label{gen:example:is:indeed:tiling}
	For $\bM,\bL$, and $T$ as in Example~\ref{AveL12:gen:example} (defined by \myeqref{AveL12:gen:example:prelim}--\myeqref{AveL12:gen:example:bL}), one has $(\bM,\bL,T) \in \cT^d$.
\end{lemma}
\begin{proof}
	We show that $\bM = \bL \oplus T$ holds; the remaining properties are easy to see. The construction of $\bL$ in Example~\ref{AveL12:gen:example} shows that the mapping $z \mapsto \sprod{a}{z}$ is a bijection from $T$ to~$\{0,\ldots,r-1\}$. This fact is used to show that every $z \in \bM=\Z^d$ is representable as $z = x+ t$ with~$x \in \bL$ and $t \in T$ in a unique way. We first verify the existence of $x$ and $t$. Using integer division of $\sprod{z}{a}$ by $r$, we write $\sprod{z}{a}$ as $\sprod{z}{a} = \ell r + m$ for suitable $\ell \in \Z$ and $m \in \{0,\ldots,r-1\}$.  There exists $t \in T$ such that $m = \sprod{t}{a}$. It follows that $z=x+t$, where $x:=z-t \in \bL$ and $t \in T$. 
	
	It remains to show that $x \in \bL$ and $t \in T$ are uniquely determined by $z$.  From $z = x + t$ we obtain $\sprod{z}{a} = \sprod{x}{a} + \sprod{t}{a}$. By the definition of $\bL$, one has $\sprod{x}{a} \in r \Z$ and, by the construction of $T$, one has $\sprod{t}{a} \in \{0,\ldots,r-1\}$. Thus, $\sprod{x}{a}/r$ is the uniquely determined quotient and $\sprod{t}{a}$ is the uniquely determined rest of the integer division of $\sprod{z}{a}$ by $r$. We have shown that $\sprod{t}{a}$ is uniquely determined by $z$. It follows that $t$ is uniquely determined by $z$, since the mapping~$z \mapsto \sprod{z}{a}$ is a bijection from $T$ to $\{0,\ldots,r-1\}$. Since~$t$ is uniquely determined by $z$, the point~$x$, too, is determined uniquely in view of $x = z-t$.
\end{proof}

The tiling in Example~\ref{AveL12:gen:example} contains nontrivially homometric pairs of lattice-convex sets, such as the one depicted in Figure~\ref{fig:exa:simplex-with-one-facet-stretched-out}. To see this, we determine some elements of $W(T,\bL)$.

\begin{lemma} \label{lem:sU_T:for:AveL12:gen:example}
	For $\bL$ and $T$ defined as in Example~\ref{AveL12:gen:example} (by \myeqref{AveL12:gen:example:prelim}--\myeqref{AveL12:gen:example:bL}), let $b_1^\ast,\ldots,b_d^\ast$ be the basis of the lattice $\bL^\ast$ as in Lemma~\ref{dual:special:lattice}\,\myeqref{lem-item:dual:special:c} and let $b_{d+1}^\ast := b_1^\ast + \cdots + b_d^\ast$. Then $\{\pm b_1^\ast, \ldots, \pm b_{d+1}^\ast\} \subseteq W(T,\bL)$.
\end{lemma}
\begin{proof} Consider $u_i := r b_i^\ast \in \Z^d$ for $i \in \{1,\ldots,d+1\}$. We need to show $w(T,u) < 1$ for each $u \in \{b_1^\ast, \ldots, b_{d+1}^\ast\}$ or, equivalently, $w(T,u_i) < r$ for each $i \in \{1,\ldots,d+1\}$. The condition $w(T,u_i)<r$ is equivalent to $\sprod{t-t'}{u_i} < r$ for all $t, t' \in T$, so let $i\in\{1,\ldots,d+1\}$ and $t,t'\in T$. Since~$\sprod{t-t'}{u_i} \in \Z$ we can reformulate $\sprod{t-t'}{u_i} < r$ as $\sprod{t-t'}{u_i} \le r-1 = kd$. Clearly,
\[
	T= \{o\} \cup \setcond{m e_1 + e_j }{m \in \{0,\ldots,k-1\},\, j \in \{1,\dots,d\}}.
\]
In view of this equality, whenever $t \ne o$, we use the representation $t = m e_1 + e_j$  with $m \in \{0,\ldots,k-1\}$ and~$j \in
\{1,\ldots,d\}$. Analogously, whenever $t' \ne o$, we use the representation $t' = m' e_1 + e_{j'}$ with $m' \in \{0,\ldots,k-1\}$ and $j' \in \{1,\ldots,d\}$.

\emph{Case~1: $t=o=t'$.} In this case, the inequality $\sprod{t-t'}{u_i} \le  kd$ is trivial.

\emph{Case~2: $i\leq d, \ t \ne o, \ t' = o$.} We need to show $\sprod{u_i}{t}
\le kd$. A direct computation shows
\begin{align*}
	\sprod{u_i}{t} = m + (j-1) k + 1 - r \sum_{\ell=i+1}^d \sprod{e_\ell}{e_j},
\end{align*}
Let~$\lambda(i,j) := \sum_{\ell = i+1}^{d}\sprod{e_\ell}{e_j}$. One has $\lambda(i,j) \in \{0,1\}$ with $\lambda(i,j) = 1$ if and only if $i< j$.  The inequality $\sprod{u_i}{t} \le kd$ is equivalent to $m + (j-1) k + 1 - r \lambda(i,j) \le k d$. The latter inequality is valid in view of $m \le k-1$, $j \le d$, and $\lambda(i,j) \ge 0$. 

\emph{Case~3: $i \leq  d, \ t \ne o, \ t' \ne o$.} 
We have to show $\sprod{t}{u_i} - \sprod{t'}{u_i} \le kd$ or, equivalently, $m -m' + (j-j') k - r (\lambda(i,j) - \lambda(i,j')) \le kd $. To prove this, we first consider the subcase $\lambda(i,j) - \lambda(i,j') \ge 0$, where the inequality follows from $m -m' \le k-1$ and $j-j' \le d-1$. In the subcase $\lambda(i,j) - \lambda(i,j') < 0$, one has $\lambda(i,j) =0$ and $\lambda(i,j') = 1$, which means $j \le i < j'$. Consequently $j-j' \le -1$ and, using $m-m' \le k-1$, we obtain the inequality in question. 

\emph{Case~4: $i \leq d, \ t =o, \ t' \ne o$.} We need to verify $\sprod{u_i}{t'} \ge -kd$ or, equivalently, $m' + (j'-1) k + 1 - r\lambda(i,j') \ge -kd$. The latter follows from $m' \ge 0$, $j' \ge 1$, $r= kd+1$, and $\lambda(i,j') \le 1$.

\emph{Case~5: $i=d+1$}.  A direct computation shows 
\(
	u_{d+1}  = \sum_{\ell=1}^d ( d - \ell + 1 ) e_\ell
\)
and hence $\sprod{u_{d+1}}{t} = m d + (d - j + 1)$ for $o\neq t= m e_1 + e_j$. Three subcases similar to Cases 2--4 can be considered and analogous arguments give the desired bounds; we omit the details. 
\end{proof}

Using Lemma~\ref{lem:sU_T:for:AveL12:gen:example} we construct pairs $S \oplus T$ and $S \oplus (-T)$ of nontrivially homometric sets for $(\bM,\bL,T) \in \cT^d$ from Example~\ref{AveL12:gen:example}. 

\begin{example}[A generalization of Example~\ref{AveL12:example}]\label{AveL12:example:generalized}
Let~$(\bM,\bL,T) \in \cT^d$ as in Example~\ref{AveL12:gen:example} (see \myeqref{AveL12:gen:example:prelim}--\myeqref{AveL12:gen:example:bL}), let $b_1^\ast,\ldots,b_d^\ast$ as in Lemma~\ref{dual:special:lattice}\,\myeqref{lem-item:dual:special:c}, and let $b_1,\ldots,b_d$ be the dual basis of $b_1^\ast,\ldots, b_d^\ast$. Using Lemma~\ref{dual:special:lattice}\,\myeqref{lem-item:dual:special:d} for the vector $a$ defined by \myeqref{AveL12:gen:example:a}, we get
\begin{equation}\label{eq:exa:generalization-of-old-exa-basis}
	b_i = 
		\begin{cases}
			(k+1) e_1 - e_2 & \text{if} \ i=1,
			\\ k e_1 + e_i - e_{i+1} & \text{if} \ 2 \le i \le d-1,
			\\ k e_1 + e_d & \text{if} \ i=d.
		\end{cases}
\end{equation}
The set $S:=\{o,b_1,\ldots,b_d\}$ is noncentrally symmetric, $\bL$-convex (because $b_1,\ldots,b_d$ is a basis of~$\bL$), and it satisfies $U(S,\bL)=\{-b_1^\ast,\ldots, -b_d^\ast,\sum_{i=1}^{d}b_{i}^{\ast}\}\subseteq W(T,\bL)$. By Theorem~\ref{thm:ABC}\,\myeqref{item:ABC:A}\,$\Rightarrow$\,\myeqref{item:ABC:B} and Proposition~\ref{prp:dir:sum:hom}, both $S \oplus T$ and $S \oplus (-T)$ are $\bM$-convex and the pair $S \oplus T$, $S \oplus (-T)$ is nontrivially homometric. Clearly, the tiling $(\bM,\bL,T)$ generates many pairs of nontrivially homometric and $\bM$-convex sets. For example, $S$ could be chosen as
\[
	\setcond{ i_1 b_1 + \cdots + i_d b_d}{ i_1,\ldots, i_d \in \{0,\ldots,N\}, \ i_1 + \cdots + i_d \le M}	
\]
with $M, N \in \N$ and $M < d N$. More generally, $S$ could be chosen as any other noncentrally symmetric, finite, $d$-dimensional, and $\bL$-convex set whose convex hull has only facets with normal vectors parallel to elements of $\{\pm b_1^\ast,\ldots, \pm b_d^\ast,\pm\sum_{i=1}^{d}b_{i}^{\ast}\}$; see Corollary~\ref{nontrivial-S-cor}.
\end{example}

Note that Example~\ref{AveL12:example:generalized} is nothing else than Example~\ref{AveL12:example} in the case $d=2$.

\begin{figure}
  \includegraphics{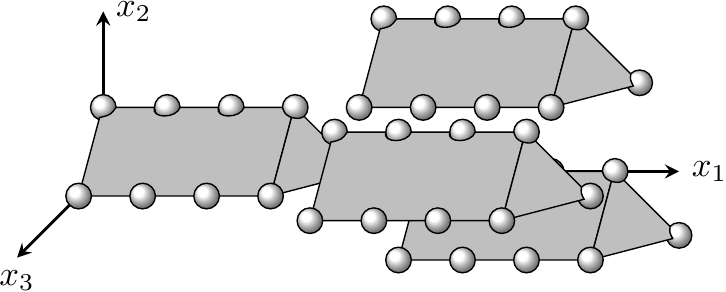}\qquad%
  \includegraphics{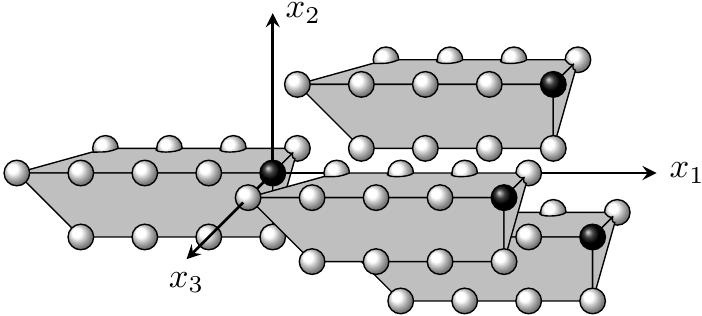}%
  \caption{A pair of nontrivially homometric $\bM$-convex sets $S \oplus T$, $S \oplus (-T)$ arising from 
Example~\ref{AveL12:example:generalized}, with $d=3$ and $k=4$. The points of $S=\{(0,0,0),(5,-1,0),(4,1,-1),(4,0,1)\}$ are drawn as black balls. The gray polytopes are the sets $s \pm \conv(T)$ with $s \in S$.}
   \label{fig:exa:simplex-with-one-facet-stretched-out}
\end{figure}%

\paragraph{Examples arising from Cartesian products.} Cartesian products can be used to generate examples of tilings, lattice-convex direct sums, and nontrivially homometric pairs for $d\geq 3$ (see also \cite[\S7]{Bianchi05} for analogous examples in the continuous setting). If, for $i \in \{1,2\}$, the subsets $K_i$ and $H_i$ of $\R^{d_i}$ ($d_i \in \N$) are homometric, then also the sets $K_1 \times K_2$ and $H_1 \times H_2$ are homometric. Further, if at least one of the two pairs $K_1, H_1$ or $K_2, H_2$ is nontrivially homometric, then also the pair $K_1 \times K_2$, $H_1 \times H_2$ is nontrivially homometric.

This construction inherits the properties that we imposed in our main results. If $(\bM_i,\bL_i,T_i) \in \cT^{d_i}$ for $i \in \{1,2\}$, then also $(\bM_1 \times \bM_2 , \bL_1 \times \bL_2, T_1 \times T_2) \in \cT^{d_1 + d_2}$. Further, two direct sums $S_i \oplus T_i$ with $i \in \{1,2\}$ generate the direct sum $(S_1 \oplus T_1) \times (S_2 \oplus T_2) = (S_1 \times S_2) \oplus (T_1 \times T_2)$ and, if $S_i \oplus T_i$ is $\bM_i$-convex, then $(S_1 \times S_2) \oplus (T_1 \times T_2)$ is $(\bM_1 \times \bM_2)$-convex. All of the above observations about the Cartesian products have straightforward proofs.

For applying our main results to examples arising from the Cartesian product, we need to analyze the set $W(T_1 \times T_2,\bL_1 \times \bL_2)$. We get the straightforward relation $W(T_1,\bL_1) \times \{o_2\} \cup \{o_1\} \times W(T_2,\bL_2) \subseteq W(T_1 \times T_2, \bL_1 \times \bL_2) \subseteq W(T_1,\bL_1) \times W(T_2,\bL_2)$, where by $o_i$ we denote the origin of $\R^{d_i}$. The latter relation follows from the equality $w(T_1 \times T_2, (u_1,u_2)) = w(T_1,u_1) + w(T_2,u_2)$ for $u_1 \in \R^{d_1}$ and $u_2 \in \R^{d_2}$. If, for  every $i \in \{1,2\}$, the set $T_i$ is $d_i$-dimensional, we even get the equality $W(T_1 \times T_2, \bL_1 \times \bL_2) = W(T_1, \bL_1) \times \{o_2\} \cup \{o_1\} \times W(T_2,\bL_2)$ in view of the inequalities $w(T,u_i) \ge \frac{1}{2}$ for $i \in \{1,2\}$ and $u_i \in \bL_i^\ast \setminus \{o_i\}$ (see Lemma~\ref{thm:if-T-covers-and-not-flat}).

\paragraph{Nontrivially homometric pairs with $\dim (T) < d$ and $\conv(S)$ having arbitrarily many facets.}

For $d\geq 3$, we construct a tiling $(\bM,\bL,T) \in \cT^d$ that contains, for each given $\ell\in\N$, a pair $S \oplus T$, $S \oplus (-T)$ of nontrivially homometric $\bM$-convex sets with $\conv(S)$ being a $d$-dimensional polytope having at least $\ell$ facets. The construction can be carried out for every $d \ge 3$, but for the sake of simplicity we stick to $d=3$ (which already gives rise to examples in higher dimensions using Cartesian products). 

\begin{example}\label{exa:parabola-construction}
	Consider $(\bM',\bL',T') \in \cT^2$ such that $T'$ is noncentrally symmetric and $W(T',\bL')$ contains two linearly independent vectors $b_1', b_2'$. For example, use the tiling from Figure~\ref{fig:exa-Dirichlet}\,(b). 
	
	We `lift' this tiling in $\cT^2$ to a tiling $(\bM, \bL, T)$ in $\cT^3$ given by $\bM = \bM' \times \Z$, $\bL  = \bL' \times \Z$, $T := T' \times \{0\}$. Changing coordinates appropriately we assume $b_1' = e_1, b_2' = e_2$; so $\bL' = \Z^2$ and $\bL = \Z^3$. In view of $w(T', e_i) < 1$ we get $w(T, (e_i,m)) = w(T', e_i) < 1$  for each~$i \in \{1,2\}$ and $m \in \Z$. 
	
	Choose any $N \in \N$ and let $P$ be the integral polygon $P$ in $\R^2$ with the $2N+1$ vertices $(i^2,i)$ lying on a parabola, where $i \in \{-N,\ldots,N\}$. Each edge of $P$ has a normal vector of the form $(1,m)$ with $m \in \Z$. Let~$S$ to be the set of all integral points in the integral polytope $[0,1] \times P$. So $\conv(S) = [0,1] \times P$ and the polytope $[0,1] \times P$ has $2N + 3$ facets: Two facets have normal vector $(1,0,0)$ and the remaining facets have normal vectors of the form $(0,1,m)$ with $m \in \Z$.  Since $w(T,(1,0,0)) = w(T', (1,0)) < 1$ and $w(T,(0,1,m)) = w(T', (0,1)) < 1$, we conclude by Theorem~\ref{thm:ABC} that $S \oplus T$ and $S \oplus (-T)$ are $\bM$-convex. Since $S$ and $T$ are not centrally symmetric, the latter is a pair of nontrivially homometric sets (see Proposition~\ref{prp:dir:sum:hom}). 
	\end{example} 
	
	This example also shows that the $d$-dimensionality assumption in Theorem~\ref{finiteness-thm} is not superfluous, because for $\dim(T)<d$, the set $W(T,\bL)$ can be infinite. 

\paragraph{Lattice-convex direct sums can have complicated summands.}\label{par:exa:complicated} 

We give `irregular' examples that do not fulfill various conditions imposed in our main results.
We start with choices for $S,T\subseteq \bM$ such that $S\oplus T$ is $\bM$-convex, but neither $S$ nor $T$ is lattice-convex with respect to any sublattice of $\bM$. A planar example is given in Figure~\ref{fig:exa-misc}\,(a); there are also examples for $d=1$:

\begin{example}\label{irreg:ex:dis1} Let~$d=1$, $\bM=\Z$, $S=\{0,1,4,5\}$ and $T=\{0,2,8,10\}$. The set $S\oplus T=\{0,1,\ldots,15\}$ is clearly $\bM$-convex. Since $1 \in S$, every lattice containing $S$ contains $\Z$. Since $2$ is an integer belonging to $\conv(S)$ but not to $S$, the set $S$ is not lattice-convex with respect to any sublattice of $\Z$. Analogously, each lattice containing $T$ contains $2\Z$. The even integer $4$ belongs to $\conv(T)$ but not to $T$, and so $T$ is not lattice-convex with respect to any sublattice of $\Z$.  
\end{example}

\begin{figure}[ht] 
   \ffigbox
  {\caption{(a) shows a $\Z^2$-convex set $S \oplus T$ such that both $S$ and $T$ are not lattice-convex with respect to any sublattice of $\Z^{2}$; (b) shows a lattice-convex set $S \oplus T$ with lattice-convex $T$, but $S$ not being lattice-convex with respect to any sublattice of $\Z^{2}$. The points in $S$ are depicted as black dots and the set $T$  is $\{(0,0),(1,0),(2,1),(1,2)\}$ and $\{0,1,2\} \times \{0,1\}$, respectively.}
  \label{fig:exa-misc}}
  {
  \begin{tabular}{p{4cm}p{4.5cm}}
    (a) \par\quad  \includegraphics{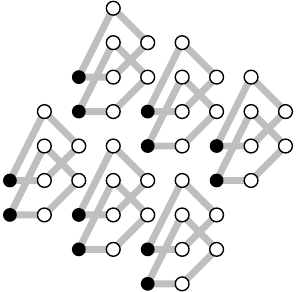} &%
    (b) \par\qquad \includegraphics{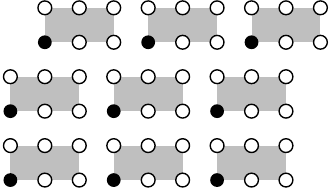} 
  \end{tabular}
  }
\end{figure}%

  There are also examples for  $S\subseteq \bM$ and $T\subseteq\bM$ such that $T$ is $\bM$-convex and $S\oplus T$ is convex, but~$S$ is not $\bL$-convex with respect to any lattice $\bL$. An example for $d=2$ is depicted in Figure~\ref{fig:exa-misc}\,(b). We also give an example in dimension three.
  
\begin{figure}[ht]
  \includegraphics{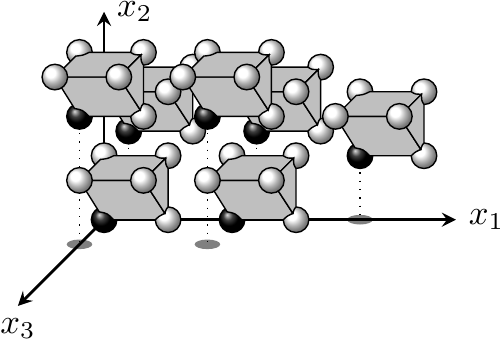}\qquad%
  \includegraphics{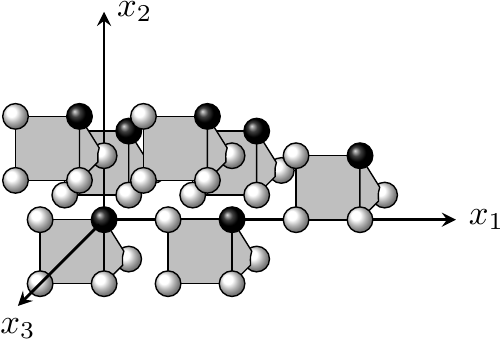}%
   \caption{A nontrivially homometric pair $S \oplus T, S \oplus (-T)$ of $\Z^3$-convex sets in $\R^3$, where $S = \{0,2\} \times \{(0,0), (1,-1), (2,1)\} \cup \{(4,1,0)\}$  and
  $ T =\{0,1\}\times\{ (0,0),(1,0),(1,1) \}$.
The set~$T$ is $\bM$-convex but $S$ is not lattice-convex with respect to any sublattice of $\bM$ (see Example~\ref{irreg:ex:dis3}). The elements of $S$ are depicted as black balls. The gray polytopes are the sets $s \pm \conv(T)$ with $s \in S$.}
   \label{fig:exa:S-not-L-convex}
\end{figure}%

  \begin{example} \label{irreg:ex:dis3}
  Let~$d=3$, $\bM = \Z^3$ and let $S$ and $T$ be given as in Figure~\ref{fig:exa:S-not-L-convex}.
  The set $S$ is not lattice-convex with respect to any sublattice of $\bM$: Each lattice containing $S$ contains $(4,1,0)$ and $(2,0,0)$, hence also $(0,1,0)=(4,1,0)-2(2,0,0)$. The point $(0,1,0)$ does not belong to $S$, but to $\conv(S)$, because $(0,1,0) = \frac{1}{3}(0,0,0) + \frac{1}{3}(0,1,-1) + \frac{1}{3}(0,2,1)$. 

  The set $T$ is clearly $\bM$-convex; $\bM$-convexity of $S\oplus T$ can be checked with Magma; see Section~\ref{sec:magma-code}. 
  One can also verify that  $S\oplus(-T)$ is $\bM$-convex, thus  $S\oplus T$, $S\oplus(-T)$ is a pair of nontrivially homometric $\bM$-convex sets (see Proposition~\ref{prp:dir:sum:hom}).
  \end{example}

Examples~\ref{irreg:ex:dis1} and~\ref{irreg:ex:dis3}  can be `lifted' to higher dimensions using Cartesian products.

\paragraph{Counterexamples to  \myeqref{item:ABC:A}\,$\Leftarrow$\,\myeqref{item:ABC:B} and \myeqref{item:ABC:B}\,$\Leftarrow$\,\myeqref{item:ABC:C} in Theorem~\ref{thm:ABC}.}\label{par:exa:ABCreverse}
In the setting of Theorem~\ref{thm:ABC}, neither \myeqref{item:ABC:A}\,$\Leftarrow$\,\myeqref{item:ABC:B} nor \myeqref{item:ABC:B}\,$\Leftarrow$\,\myeqref{item:ABC:C} holds in general, as the following examples show. 

\begin{example}[Counterexample to \myeqref{item:ABC:A}\,$\Leftarrow$\,\myeqref{item:ABC:B} in Theorem~\ref{thm:ABC}] \label{count:ex:to:a<=b}
	Let~$d \ge 3$, $\bL:= \Z^d$, $\bM = \Z^d + \Z a$ with~$a:=\frac{1}{2}\sum_{i=1}^{d}e_{i}$, $T := \{o,a\}$, and $S := \{o,e_1,\ldots,e_{d}\} \subseteq \bL$. In view of $2 a \in \Z^d$, we get~$\bM = \bL \oplus T$. The assumptions of Theorem~\ref{thm:ABC} are clearly fulfilled. 

The set $\Delta=\conv(\{e_{1},\ldots,e_{d}\})$ is a facet of $\conv(S)$ and $u:=(1,\ldots,1)\in\Z^{d}=\bL^{*}$ is a primitive outer normal of $\Delta$. We have $w(T,u)= \frac{d}{2} \geq1$, so \myeqref{item:ABC:A} from Theorem~\ref{thm:ABC} does not hold.

	We verify the validity of condition \myeqref{item:ABC:B} from Theorem~\ref{thm:ABC}. Let~$x\in \conv(S \oplus T)\cap \bM$. We have to show that $x\in S\oplus T$. One has  $\conv(S \oplus T) = \conv(\{o,e_1,\ldots,e_{d-1}\}) + [o,a]$, so there exist $\lambda_{1},\ldots,\lambda_{d},\mu\geq 0$ such that $\sum_{i=1}^{d}\lambda_{i}\leq 1$, $\mu\leq 1$, and
        $
          x=\sum_{i=1}^{d}\lambda_{i}e_{i} + \mu a = \sum_{i=1}^{d}\left(\lambda_{i}+\frac{1}{2}\mu\right)e_{i}
        $. 
By definition of~$\bM$ and since $x\in\bM$, we either have $\lambda_{i}+\frac{1}{2}\mu\in\Z$ for each~$i\in\{1,\ldots,d\}$ (case~1), or we have $\lambda_{i}+\frac{1}{2}\mu\in\frac{1}{2}\Z\setminus\Z$ for each~$i\in\{1,\ldots,d\}$ (case~2).

We analyze the first case: If $\mu\neq0$, the case assumption gives $\lambda_{i}\geq\frac{1}{2}$ for each~$i\in\{1,\ldots,d\}$. Employing $d\geq 3$, the latter fact contradicts $\sum_{i=1}^{d}\lambda_{i}\leq 1$. So $\mu=0$ and hence $\lambda_{i}\in\{0,1\}$ for each~$i\in\{1,\ldots,d\}$. In view of $\sum_{i=1}^{d}\lambda_{i}\leq 1$ we conclude that $x\in S$, which gives the assertion.

Now we analyze the second case: If $\mu=0$, the case assumption gives $\lambda_{i}\geq\frac{1}{2}$ for each~$i\in\{1,\ldots,d\}$, which contradicts $\sum_{i=1}^{d}\lambda_{i}\leq 1$, again using $d\geq 3$. So $0<\mu\leq 1$. If $\mu=1$, then $\lambda_{i}\in\{0,1\}$ for each~$i\in\{1,\ldots,d\}$ and in view of $\sum_{i=1}^{d}\lambda_{i}\leq 1$ we conclude that $x\in S\oplus T$, giving the assertion in this subcase. We are left with the subcase $0<\mu<1$, in which the bounds on $\lambda_{i}$ and the case assumption give $\lambda_{i}+\frac{1}{2}\mu=\frac{1}{2}$ for each~$i\in\{1,\ldots,d\}$. But this means $x=\sum_{i=1}^{d}\frac{1}{2}e_{i}=a$, hence $x\in S\oplus T$.
\end{example}

\begin{example}[Counterexample to \myeqref{item:ABC:B}\,$\Leftarrow$\,\myeqref{item:ABC:C} in Theorem~\ref{thm:ABC}]\label{count:ex:to:b<=c}
 Let~$d\geq3$, $\bM=\Z^d$, $\bL=d\Z^d$,  $S:=\{o,de_1,\ldots,de_{d}\}$, and $T:=\{0,1,\ldots,d-1\}^d$. Clearly, $(\bM,\bL,T) \in \cT^d$. For each facet~$F$ of $\conv(S)$, the relation $\aff(F) \subseteq \conv(F) + \bL$ involved in \myeqref{item:ABC:C} does not hold, because it is  equivalent to the obvious relation $\R^{d-1} \nsubseteq \conv(\{o,e_1,\ldots,e_{d-1}\}) + \Z^{d-1}$ in dimension $d-1 \ge 2$. Thus, \myeqref{item:ABC:C} holds, as \myeqref{item:ABC:C} imposes nothing on facets $F$ of $\conv(S)$ which do not satisfy $\aff(F) \subseteq \conv(F) + \bL$. On the other hand, \myeqref{item:ABC:B} does not hold, that is, $S \oplus T$ is not $\bM$-convex. In fact, the point $d\sum_{i=1}^{d}e_i$ does not belong to $S \oplus T$ but to $\conv(S \oplus T)$, as this point can be written as the convex combination $\sum_{i=1}^d \frac{1}{d} ( d e_i + (d-1)\sum_{j=1}^{d}e_j)$ of the points $ d e_i + (d-1)\sum_{j=1}^{d}e_j\in S \oplus T$ with $i \in \{1,\ldots,d\}$.
\end{example}

Further counterexamples to \myeqref{item:ABC:A}\,$\Leftarrow$\,\myeqref{item:ABC:B} and \myeqref{item:ABC:B}\,$\Leftarrow$\,\myeqref{item:ABC:C} can be constructed using Examples~\ref{count:ex:to:a<=b} and~\ref{count:ex:to:b<=c}  and taking Cartesian products.

\section{Magma code}\label{sec:magma-code}

We present the Magma code that we used to perform several computations in the context of this manuscript, including the computer enumeration used in Theorem~\ref{thm:complete-classification}. Note that the Magma engine is also available online as the so-called Magma calculator at \url{http://magma.maths.usyd.edu.au/calc/} (currently using version V2.21-8 and restricting the running time to two minutes).

\paragraph{Basic computations.}\label{par:magma:convexity}

For computations in dimension $d \in \N$ we introduce the so-called toric lattice, where the dimension $d$ has to be specified (e.\,g.\ by \lstinline!d:=3;! for the example below):
\begin{lstlisting}
V:=ToricLattice(d);
\end{lstlisting}
In the context this paper, $V$ can be understood as the vector space $\rational^d$ over the field $\rational$;  Magma procedures related to lattices and invoked on objects `living' in $V$ are carried out with respect to the lattice~$\Z^d$.

We can now introduce subsets of $V$. For example, sets $S$ and $T$ in Figure~\ref{fig:exa:S-not-L-convex} can be given by 
\begin{lstlisting}
S:={ V | [0,0,0],[0,1,-1],[0,2,1],[2,0,0],[2,1,-1],[2,2,1],[4,1,0] };
T:={ V | [0,0,0],[1,0,0],[0,1,0],[1,1,0],[0,1,1],[1,1,1] };
\end{lstlisting}

Whether a given set is $\Z^d$-convex can be tested using the function 
\begin{lstlisting}
IsLatticeConvex:=function(K)
   return Points(Polytope(K)) eq K;
end function;
\end{lstlisting}
which compares $\Z^d \cap \conv(K)$ with $K$. Testing whether the sum of $S$ and $T$ is direct can be implemented by comparing the cardinalities of $S$, $T$, and $S+T$: 
\begin{lstlisting}
IsSumDirect:=function(S,T)
   return #S * #T  eq  #{s+t : s in S, t in T};
end function;
\end{lstlisting}
Now calling \lstinline!IsLatticeConvex({s+t : s in S, t in T});! and  \lstinline!IsSumDirect(S,T);! checks lattice-convexity of $S+T$ and tests if the sum of $S$ and $T$ is direct, respectively.

\paragraph{Computing $W(T,\bL)$.}\label{par:magma:UT} For the computation of $W(T,\bL)$ we use the representation $W(T,\bL) = (\bL^\ast \cap \intr(D(T)^\circ)) \setminus \{o\}$. Let~$b_1,\ldots,b_d \in \Q^d$ be linearly independent and let $\bL$ be  the lattice with basis $b_1,\ldots,b_d$. Then $\bL=B(\Z^d)$, where $B$ is the matrix with columns $b_1,\ldots,b_d$. By Proposition~\ref{change:prp}, we get $W(T,\bL) = (B^{-1})^\top \bigl( W(B^{-1}(T), \Z^d) \bigr)$. We use the latter representation for the computation of $W(T,\bL)$. Note that, in Magma, converting the list of vectors $b_1,\ldots,b_d$ to a matrix we get a matrix whose \emph{rows} are $b_1,\ldots,b_d$. Therefore, in the following code, $B$ is initially the list of vectors $b_1,\ldots,b_d$ and, during the computation, $B$ is converted to the matrix with rows $b_1,\ldots,b_d$. In correspondence to this, the elements of $\Q^d$ are interpreted as rows. We also remark that, for a rational polytope $P$ in $\R^d$, the Magma expressions \lstinline!InteriorPoints(P)!, \lstinline!Polar(P)!, and \lstinline!P+(-P)! compute $\Z^d \cap \intr(P)$, $-P^{\circ}$, and  $D(P)$, respectively.  

\begin{lstlisting}[firstnumber=1]
Q:=RationalField(); // the field of rational numbers

W:=function(T,B)
    B:=Matrix(Q,B); // convert B into a matrix
    T:=Matrix(Q,T); // convert T into a matrix
    T:=Rows(T*B^(-1)); // replace T by its image under B^(-1)
    P:=Polytope([ V | Eltseq(t) : t in T ]); // P is conv(T)
    _W_:=InteriorPoints(Polar(P+(-P))) diff {Zero(Dual(V))}; 
    _W_:=Matrix(Q,[Eltseq(z) : z in _W_]); // convert _W_ to a matrix
    _W_:=_W_*Transpose(B^(-1)); // _W_ for the original choice of T
    _W_:=[Dual(V) | Eltseq(z) : z in Rows(_W_) ]; // _W_ as a list
    return _W_;
end function;
\end{lstlisting}

We illustrate how to use the above Magma function by an example. Let~$d=3$ and let $T$ be as in Figure~\ref{fig:exa:simplex-with-one-facet-stretched-out}, so the vertices of $\conv(T)$ can be passed to Magma as follows:

\begin{lstlisting}
vertT:=[ [0,0,0],[4,0,0],[0,1,0],[0,0,1],[3,1,0],[3,0,1] ];
\end{lstlisting}

The lattice $\bL$ can be given by the following base $b_1,b_2,b_3$ obtained from~\myeqref{eq:exa:generalization-of-old-exa-basis} on page~\pageref{eq:exa:generalization-of-old-exa-basis}:
\begin{lstlisting}
B:=[[5,-1,0],[4,1,-1],[4,0,1]];
\end{lstlisting}

Now the set $W(T,\bL)$ can be evaluated using the expression \lstinline!W(vertT,B);!

\paragraph{The enumeration procedure in the proof of Theorem~\ref{thm:complete-classification}.} We present a complete implementation of the enumeration procedure used in the proof of Theorem~\ref{thm:complete-classification}.

We introduce a two-dimensional toric lattice:
\begin{lstlisting}[firstnumber=1]
V:=ToricLattice(2);
\end{lstlisting}

The following code computes the base $b_1= (\ell,0), b_2= (s,h)$ of $\bL$  from the input $\ell$, $h$, and $s$:
\begin{lstlisting}[firstnumber=last]
Base:=function(l,h,s)
   return [ V | [l,0],[s,h]];
end function;
\end{lstlisting}

The corresponding dual base $b_1^\ast,b_2^\ast$ can be computed as follows:

\begin{lstlisting}[firstnumber=last]
DualBase:=function(l,h,s)
   return [ Dual(V) | [1/l,-s/(l*h)],[0,1/h]]; 
end function;
\end{lstlisting}

Given the parameters $\ell$, $h$, $s$, and $q=(q_1,q_2)$, the following functions generate the tiles $T_{q}=\Z^{2}\cap\sum_{i=1}^2 [ (q_i + n_i)/L, (q_i+L)/L ] b_i$ from Lemma~\ref{lem-enum-T-given-bL}. We compute the two segments $[(q_i + n_i)/L, (q_i+L)/L ]b_{i}$ with $i \in \{1,2\}$, take their Minkowski sum and, eventually, return the convex hull of the integral points in the Minkowski sum (using \lstinline!IntegralPart!):
\begin{lstlisting}[firstnumber=last]
SpecialSegment:=function(q,n,L,b)
   return Polytope( [(q+n)/L*b, (q+L)/L*b] );
end function;

Tile:=function(l,h,s,q1,q2,n1,n2)
   L:=l*h;
   b:=Base(l,h,s);
   I1:=SpecialSegment(q1,n1,L,b[1]);
   I2:=SpecialSegment(q2,n2,L,b[2]);
   return IntegralPart(I1+I2);
end function;
\end{lstlisting}

The following procedure performs a search of the relevant tiles $T$ arising from the base defined by the given parameters $\ell,h,s$. The possible sets $T$ are filtered with respect to three criteria: $T$ must be two-dimensional, the condition $w(T,b_1^\ast+b_2^\ast) < 1$ must hold, and $w(T,\bM)\geq 2$ must hold, since in the case $w(T,\bM)=1$ all relevant tilings $\Z^2 = T \oplus \bL$ have already been classified in~\cite{AveL12}. We use \lstinline!Width(P,u)! (available in Magma V2.21-8, currently used in the Magma calculator) to compute the width function of a polytope $P$ for direction $u$. The expression \lstinline!Width(P)! evaluates the lattice width of a polytope $P$ with respect to the integer lattice.
Note that in terms of $\ell ,h,s$, the vector $n=(n_1,n_2)$ from Lemma~\ref{lem-enum-T-given-bL} is defined by $n_1=\gcd(h,s)$ and $n_2=\ell$. (The functions \lstinline!IndentPush! and \lstinline!IndentPop! merely produce indented output.)
\begin{lstlisting}[firstnumber=last]
SearchForTilesWithBase:=procedure(l,h,s)
   b:=Base(l,h,s);
   print "Searching for tiles arising from the base", b;
   L:=l*h;
   db:=DualBase(l,h,s);
   DiagonalDirection:=db[1]+db[2];
   n1:=GCD(h,s);
   n2:=l;
   for q1:=0 to L-1 by n1 do
      for q2:=0 to L-1 by n2 do
         T:=Tile(l,h,s,q1,q2,n1,n2);
         if 
            IsMaximumDimensional(T) 
            and Width(T,DiagonalDirection) lt 1 
            and Width(T) gt 1
         then
            IndentPush();
            print "Found tile with vertices ", Vertices(T);
            IndentPop();
         end if;
      end for;
   end for;
end procedure;
\end{lstlisting}

The following procedure searches the associated sets $T$ for all relevant bases, the latter defined by triples $\ell,h,s$ and with a given determinant $L$. Since $h \le 2$ yields a lattice with  $\Delta=\conv(\{o,b_1,b_2\})$ satisfying~$w(\Delta,\Z^{2})\leq 2$,  the condition $h\geq 3$ is included to ensure $w(\Delta,\Z^{2})\geq 3$.  
\begin{lstlisting}[firstnumber=last]
SearchForBasesWithDet:=procedure(L)
   print "Searching for bases with determinant", L;
   for l in Divisors(L) do
      h:=L div l;
      if h ge 3 then
         for s in [0..h-1] do
            Delta:=Polytope([[0,0],[l,0],[s,h]]);
            if 
               ( 12 le L and L le 16 and Width(Delta) eq 4 )
               or
               (  7 le L and L le 18 and Width(Delta) eq 3 )
            then
               IndentPush();
               SearchForTilesWithBase(l,h,s);
               IndentPop();
            end if;
         end for;
      end if;
   end for;
end procedure;
\end{lstlisting}

Finally, the following code searches for the relevant $(\Z^2,\bL, T)\in\cT^{2}$ with  $\det(\bL)\in\{7,\ldots,18\}$:
\begin{lstlisting}[firstnumber=last]
for L in [7..18] do
   IndentPush();
   SearchForBasesWithDet(L);
   IndentPop();
end for;
\end{lstlisting}

The running time of the code on a currently regular desktop  computer is about 3 minutes. A~reader willing to check our results may use the Magma calculator. To circumvent the time limit of two minutes for the Magma calculator, it is reasonable to run the code twice, replacing the whole range $\{7,\ldots,18\}$ for the determinant of $\bL$ by two ranges, say $\{7,\ldots,15\}$ and $\{16,\ldots,18\}$. For both of these ranges, the Magma calculator finishes the computation in about 100 seconds.

\paragraph{Acknowledgement.} We thank the anonymous referees for valuable suggestions regarding a previous version of this paper.

\bibliographystyle{amsalpha}
\bibliography{lit}
\end{document}